\documentclass[11pt]{article}
\usepackage{graphicx}
\usepackage[top=2cm,left=2cm,right=2cm,bottom=2cm]{geometry}
\usepackage[utf8]{inputenc}
\usepackage[english]{babel}  
\usepackage{amsfonts}
\usepackage{amsmath}
\usepackage{amssymb}
\usepackage{amsthm}
\usepackage{epstopdf}
\usepackage{bbold}
\usepackage{ wasysym }
\usepackage{array}
\usepackage{ textcomp }
\usepackage{hyperref}
\usepackage{hhline}
\usepackage{stmaryrd}
\usepackage{xcolor}

\newcommand{\Pro}{\mathbb{P}}
\newcommand{\R}{\mathbb{R}}

\newcommand{\beq}{\begin{equation}}
\newcommand{\eeq}{\end{equation}}
\newcommand{\bepa}{\left\{ \begin{array}{l}}
\newcommand{\eepa}{\end{array} \right.}
\newcommand{\p}{\partial}
\newcommand{\f}{\frac}
\newtheorem{theorem}{Theorem}
\newtheorem{lemma}{Lemma}
\newtheorem{definition}{Definition}
\newtheorem{remark}{Remark}
\newtheorem{proposition}{Proposition}
\newtheorem{corollary}{Corollary}

\begin{document}

\title{Quantifying the Survival Uncertainty of \textit{Wolbachia}-infected Mosquitoes in a Spatial Model
}
\author{Martin Strugarek\thanks{AgroParisTech, 16 rue Claude Bernard, F-75231 Paris Cedex 05 
              \& LJLL, UPMC, 5 place Jussieu, 75005 Paris France \href{mailto:strugarek@ljll.math.upmc.fr}{strugarek@ljll.math.upmc.fr}}
\and
        Nicolas Vauchelet\thanks{LAGA - UMR 7539
 Institut Galil\'{e}e
 Universit\'{e} Paris 13
 99, avenue Jean-Baptiste Cl\'{e}ment
 93430 Villetaneuse - France}
\and
        Jorge P. Zubelli\thanks{IMPA, Estrada Dona Castorina, 110
 	      Jardim Botânico
 	      22460-320
 	      Rio de Janeiro, RJ - Brazil}
}

\maketitle

\begin{abstract}
Artificial releases of \textit{Wolbachia}-infected \textit{Aedes} mosquitoes have been under study in the past years
for fighting vector-borne diseases such as dengue, chikungunya and zika.
Several strains of this bacterium cause cytoplasmic incompatibility (CI) and can also affect their host's
 fecundity or lifespan, while highly reducing vector competence for the main arboviruses.

We consider and answer the following questions: 1) what should be the initial condition ({\it i.e.} size of the initial mosquito population) to have invasion with one mosquito release source? We note that it is hard to have an invasion in such case. 2) How many release points does one need to have sufficiently high probability of invasion? 3) What happens if one accounts for uncertainty in the release protocol ({\it e.g.} unequal spacing among release points)?

We build a framework based on existing reaction-diffusion models for the uncertainty quantification in this context,
obtain both theoretical and numerical lower bounds for the probability of release success
and give new quantitative results on the one dimensional case.
\end{abstract}

\section{Introduction}

In recent years, the spread of  chikungunya, dengue, and zika has become a
major public health issue, especially in tropical areas of the planet
\cite{CDC,Bha.Burden}.  All those diseases are caused by arboviruses whose main transmission vector is the 
\textit{Aedes aegypti}. 
One of the most important and innovative ways of vector control 
 is the artificial introduction of a maternally transmitted
bacterium of genus \textit{Wolbachia} in the mosquito population (see \cite{Bla.Wolbachia,Joh.Impact,Wal.wMel}). 
This process has been successfully implemented on the field (see \cite{Hof.Successful}).
It requires the release of \textit{Wolbachia}-infected mosquitoes on the field and ultimately depends on the
prevalence of one sub-population over the other.
Other human interventions on mosquito populations may require such spatial release protocols 
(see \cite{Alp.Genetic,Alp.Aedes} for a review of past and current field trials for genetic mosquito population modification).
Designing and optimizing these protocols remains a challenging problem for today 
(see \cite{Han.Strategies,Vav.Making}), and may be enriched by the lessons learned from previous release 
experiments (see \cite{Hof.stability,Ngu.Field,Yea.Mitochondrial})

This article studies a spatially distributed model for
the spread of \textit{Wolbachia}-infected mosquitoes in a population and its
success as far as non-extinction probabilities are concerned. 
We address the question of the release protocol to guarantee a high probability of invasion. More precisely, what quantity of mosquitoes need to be released to ensure invasion, if we have only one release point? What if we have multiple release points and if there is some uncertainty in the release protocol?
We obtain lower bounds so as to quantify the  success probability  of spatial  spread of the introduced population according to a mathematical model.  

We define here an \textit{ad hoc} framework for the computation of this success probability.
As a totally new feature added to the previous works on this topic 
(see \cite{Cra.Wolbachia,Han.Modelling,Han.Population,Jan.Stochastic,Tur.Cytoplasmic,Yea.Dynamics}), it involves space variable as a key ingredient.
In this paper we provide quantitative estimate and numerical results in dimension $1$.

It is well accepted that stochasticity plays a significant role in biological modeling. 
Probabilities of introduction success have already been investigated for genes or other agents into a wild biological population. 
The recent work \cite{BarTur.Spatial} makes use of reaction-diffusion PDEs to describe the biological phenomena underlying sucessful introduction as cytoplasmic analogues of the Allee effect.
The infection of the mosquito population by {\em Wolbachia} is seen as an ``alternative trait'', spreading across a population having initially a homogeneous regular trait.
Other recent models have been proposed either to compute the invasion speed 
(\cite{chankim}), 
or get an insight into the induced time dynamics of more complex systems, including humans or pathogens 
(see \cite{Fen.Solving,HugBri.Modelling}).
In the mosquito part, models usually feature two stable steady states: invasion (the regular trait disappears) and extinction (the alternative trait disappears). 
Since this phenomenon is currently being investigated as a tool to fight \textit{Aedes} transmitted diseases, the problem of determination of thresholds for invasion in this equation is of tremendous importance. 

The issue of survival probability of invading species has attracted a lot of attention by many researchers. Among such we may cite \cite{BarRou} and \cite{RouBar}. We stress, however, that this is not the direction followed in this paper. In the cited articles indeed, the basic underlying model is either a stochastic PDE or its discretization, and the uncertainty concerning the initial state is not considered.

In other words, although in a deterministic model as ours one can in principle numerically check for a specific initial configuration whether the invasion by the {\it Wolbachia}-infected mosquitoes will be successful or not, in practice such a specific initial condition is subject to uncertainty, and therefore the uncertainty quantification of the success probability is a natural question.

Our modeling goes as follows: We consider a domain $\Omega$, 
a frequency $p : \Omega \to [0, 1]$ that models the prevalence of the \textit{Wolbachia} infection trait. 
More specifically, in the case of cytoplasmic incompatibility caused in \textit{Aedes} mosquitoes by the endo-symbiotic bacterium \textit{Wolbachia}, $p$ is the proportion of mosquitoes infected by the bacterium ({\it e.g.} $p=1$ means that the whole population is infected).
Then, this frequency obeys a bistable reaction-diffusion equation.
We aim at estimating the invasion success probability with respect to the initial data (= release profile).

In \cite{BarTur.Spatial,sv2016} it was obtained
an expression for the reaction term $f$ in the limit Allen-Cahn equation 
\begin{equation}
\label{AllenCahn}
  \partial_t p - \sigma \Delta p = f(p) 
\end{equation}
in terms of the following biological parameters:
$\sigma$ diffusivity (in square-meters per day, for example), $s_f$ (effect of \textit{Wolbachia} on fecundity, $=0$ if it has no effect); $s_h$ (strength of the cytoplasmic incompatibility, $=1$ if it is perfect);
$\delta$ (effect on death rate, $d_i = \delta d_s$ where $d_s$ is the regular death rate without \textit{Wolbachia}) and $\mu$ (imperfection of vertical transmission, expected to be small).
It reads as follows: 
\begin{equation}
  f(p) = \delta d_s p \frac{-s_h p^2 + \big( 1+s_h - (1-s_f)(\frac{1 - \mu}{\delta} + \mu) \big) p + (1-s_f) \frac{1-\mu}{\delta} - 1}{s_h p^2 - (s_f + s_h)p + 1}.
\end{equation}
Bistable reaction terms are such that $f < 0$ on $(0, \theta)$ and $f > 0$ on $(\theta, \theta_+)$. Usually, we consider $\theta_+ = 1$. This is the case if $\mu = 0$.

The outline of the paper is the following. 
In the next section, we state and prove our main result: the existence
of compactly radially symmetric functions such that if the initial data
is above one of such function, then invasion occurs.
Then we explain how this result provides an estimate of the probability of
success of a release protocol.
Section \ref{sec:1D} and the following is devoted to the one dimensional case.
In Section~\ref{sec:4} we provide an analytical computation of the probability of success.
Numerical results are displayed in Section \ref{sec:num}. 
We conclude in Secion~\ref{sec:conclusion}.
Finally an appendix is devoted to the study of the minimization of the 
perimeter of release in one dimension.

\section{Setting the problem: How to use a threshold property to design a release protocol?}

\subsection{The threshold phenomenon for bistable equations}

In Equation~\eqref{AllenCahn}, we assume that 
\beq
\bepa
\exists \, \theta\in (0,1),\, f(0)=f(\theta)=f(1)=0, 
\\[10pt] f<0 \mbox{ on } (0,\theta),
\quad f>0 \mbox{ on } (\theta,1), \quad \int_0^1 f(x) dx > 0.
\label{hyp:bistab}
\eepa
\eeq
A consequence of this hypothesis is the existence of invading traveling waves.
From now on, we denote $F$ the anti-derivative of $f$ which vanishes at $0$,
\begin{equation}
F(x):=\int_0^x f(y)\,dy.
\label{def:F}
\end{equation}
Since we have assumed $F(1)>0$, by the bistability of the function $f$, there exists a unique 
$\theta_c\in(0,1)$ such that 
$$
F(\theta_c)=\int_0^{\theta_c} f(x) dx = 0.
$$

By making use of biologically reasonable parameters ($d_s = 0.27$, $s_f= 0.1$, $\mu=0$, 
$s_h=0.8$ and $\delta=0.3/0.27 = 10/9$), we obtain the profiles for $f$ and its anti-derivative in Figure \ref{fig:reacprofiles}. 
In \cite{HugBri.Modelling}, the authors used the notations $\phi = 1 - s_f$, $\delta = d_i / d_s$, $u = 1 - s_h$ and $v= 1 - \mu$.
They gave a range of values of these parameters for three \textit{Wolbachia} strains, 
namely \textit{wAlbB}, which has no impact on death ($\delta = 1$) but reduces fecundity, 
\textit{wMelPop} which highly increases death rate but isn't detrimental to fecundity, 
and \textit{wMel} which has a moderate impact on both.
Values are given in Table 3 of the cited article (which contains also a parameter $r$, 
standing for differential vector competence of \textit{Wolbachia}-infected mosquitoes for dengue, 
a feature we do not include in our modelling since we focus on the mosquito population dynamics), 
see the references therein for more details.
According to the aforementioned references, the authors always assumed perfect CI
and maternal transmission, that is, with our notations $s_h=1$ and $\mu=0$.
Our notations mimic those of \cite{BarTur.Spatial,Fen.Solving}, 
where they did not give as detailed tables for the parameters as in \cite{HugBri.Modelling}, although we refer the reader 
to the references they gave, which contain some quantitative estimations of these parameters.
Our choices for $d_s$ and $s_f$ reflect the field data exposed in~\cite{Dut.Lab}, for the (life-shortening) \textit{wMel} strain in the context of the city of Rio de Janeiro, in Brazil.
\begin{figure}[h]
 \includegraphics[width=\textwidth]{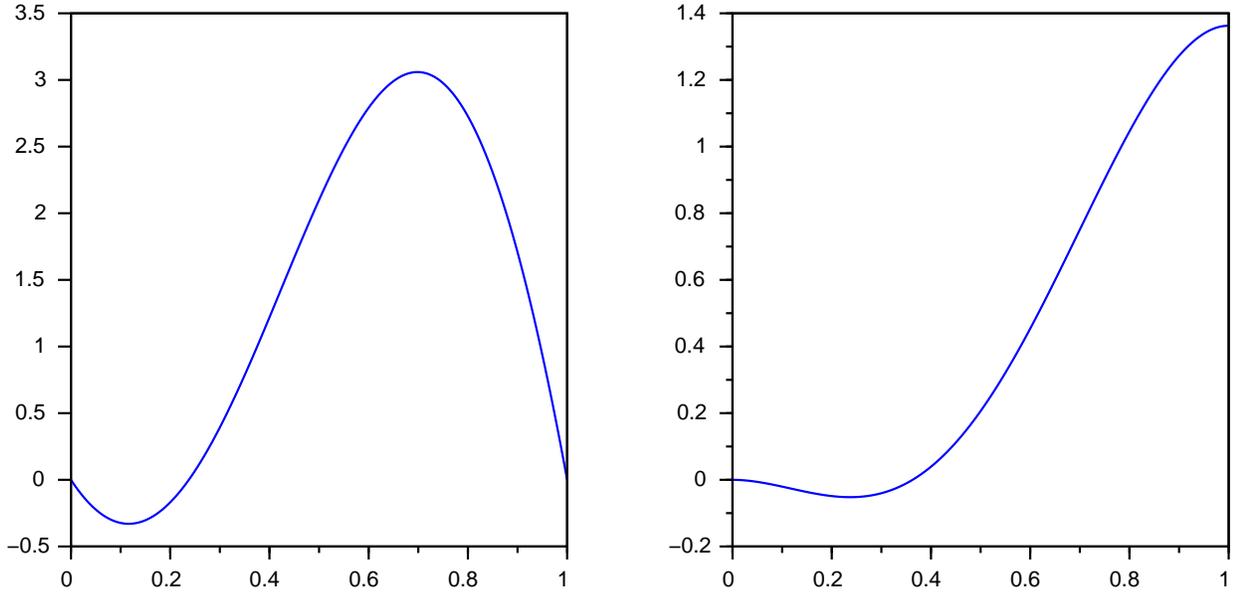}
 \caption{Profile of $f$ (left) and of its anti-derivative $F$ (right).}
 \label{fig:reacprofiles}
\end{figure}

We will always assume $\mu=0$ (perfect vertical transmission) in the following.
Note that our results also apply when $\mu > 0$, but in this case the ``invasion'' state is not exactly $p=1$, but $p=p_+ (\mu) < 1$, because there is a flaw in \textit{Wolbachia} vertical (=maternal) transmission.

Moreover, following estimates from \cite{Dut.Lab,Vil.Bayesian} for \textit{Aedes aegypti} in Rio de Janeiro (Brazil), and general literature review and discussion in Section 3 of \cite{Otero} we consider that mosquitoes spread at around $\sigma = 830 \mathrm{m}^2 / \text{day}$ (see the references given in \cite{Otero} for more details).
With these estimations of the parameters, the quantitative results we get are satisfactory because they appear to be relevant for practical purposes.
For example, in order to get a significant probability of success, the release perimeter we find is around $595 \mathrm{m}$ wide (in one dimension).
In the example from Figure~\ref{fig:reacprofiles}, $\theta_c \simeq 0.36$.

We say that a radially symmetric function $\phi$ on $\R^d$ is non-increasing 
if $\phi(x)=g(|x|)$ for some $g$ that is non-increasing on $\R^+$.

The following result gives a criterion on the initial data to guarantee invasion.
\begin{theorem}
Let us assume that $f$ is bistable in the sense of \eqref{hyp:bistab}.
Then, for all $\alpha \in (\theta_c, 1]$ there exists a compactly supported, radially symmetric non-increasing function 
 $v_\alpha(|x|)$, with $v_\alpha:\R_+ \to \R_+$ non-increasing, $v_\alpha (0) = \alpha$ (called ``$\alpha$-bubble''),
such that if $p$ is a solution of 
\begin{align}
& \p_t p - \sigma \Delta p = f(p), \label{eq:p} \\
& p(t=0,x) = p^0(x) \geq v_{\alpha}(|x|), \nonumber
\end{align}
then $p \xrightarrow[t \to \infty]{} 1$ locally uniformly.
Moreover, we can take Supp$(v_\alpha)=B_{R_\alpha}$ with
\begin{equation}\label{def:Ralpha}
R_{\alpha} = \sqrt{\sigma} \inf_{\rho \in \Gamma} \sqrt{ \f{1-\rho^d}{(1-\rho)^2} \f{1}{\big( \int_0^{\alpha} \big( 1 - \f{1-\rho}{\alpha} x \big)^d f(x) dx \big)_+} },
\end{equation}
where $\Gamma=\{\rho\in(0,1),\ \int_0^\alpha (1-\frac{1-\rho}{\alpha}x)^d f(x)dx >0\}$.

In one dimension, we have the sharper estimate Supp$(v_\alpha)=[-L_\alpha,L_\alpha]$ with
\begin{equation}\label{def:lalpha}
 L_{\alpha} = \sqrt{\f{\sigma}{2}} \int_0^{\alpha} \f{dv}{\sqrt{F(\alpha) - F(v)}}.
\end{equation}
\label{thm:invasion}
\end{theorem}

\begin{remark}
 Clearly, the set $\Gamma$ is nonempty. Indeed for $\rho \sim 1$, 
 \[
\int_0^\alpha (1-\frac{1-\rho}{\alpha}x)^d f(x)dx>0, 
 \]
since $F(\alpha) > 0$.   
 However, it is hard to say more unless we consider a specific function $f$.
\end{remark}

Theorem \ref{thm:invasion} is a well-known fact 
(see \cite{DuMat.Convergence,MatPol.Dynamics,MurZho.Threshold,Pol.Threshold,Zla.Sharp}), 
even though the explicit formulae \eqref{def:Ralpha}, \eqref{def:lalpha} are seldom found in the literature.
We postpone to Section~\ref{sub:thmproof} the proof of this result, which follows essentially from the ideas developped in~\cite{MurZho.Threshold}.

We recall the definition of a ``ground state'' as a positive stationary solution $v$ of \eqref{AllenCahn}, \textit{i.e.}:
$$
-\Delta v = f(v)
$$
that decays to $0$ at infinity.
In dimension $d=1$ (and in some special cases in higher dimensions, see \cite{MurZho.Threshold}), such a ground state is unique up to translations.
When $d=1$ we denote $v_{\theta_c}$ the ground state which is maximal at $x=0$.
It is symmetric decreasing and $v_{\theta_c} (0) = \theta_c$, which is consistent with the notation $v_{\alpha}$ in Theorem \ref{thm:invasion}.
Although we won't use it in the rest of the paper, we note that with a similar argument, we have a sufficient condition for the extinction:

\begin{proposition}
 In dimension $d=1$, let $p$ be a solution of equation~\eqref{AllenCahn}, associated with the initial value $p_0$.
 If $p_0 < \theta$ or $p_0 < v_{\theta_c}(\cdot - \zeta)$ for some $\zeta \in \R$,
 then $p$ goes extinct: $p \xrightarrow[t \to \infty]{} 0$ uniformly on $\R$.
 \label{prop:extinction}
\end{proposition}

\subsection{The stochastic framework for release profiles}
\label{stochasticframework}
When releasing mosquitoes in the field, the actual profile of \textit{Wolbachia} infection
in the days right after the release is very uncertain.
In order to quantify this uncertainty, we define in this section an adequate space of release profiles.
The pre-existing mosquito population is assumed to be homogeneously dense, at a level $N_0 \in \R_+$.

From now on, we assume that we have fixed a space unit, so that we may talk of numbers or densities of mosquitoes without any trouble.

We define a spatial process $X_{\cdot} (\omega)=X(\cdot,\omega) : \R^d \to \R_+$ as the introduced mosquitoes profile. 

We expect that the time dynamics of the infection frequency will be given by \eqref{AllenCahn}
\beq
\bepa
\p_t p -\sigma  \Delta p = f(p),
\\[10pt]
p(t=0, \tau; \omega) = \displaystyle \f{X_{\tau} (\omega)}{X_{\tau} (\omega) + N_0}.
\eepa
\eeq
We want to measure the probability of establishment associated with this set 
of initial profiles.

Making use of Theorem \ref{thm:invasion}, we want to give a lower bound for the probability of non-extinction 
(which is equivalent to $1$ minus the probability of invasion, by the sharpness of threshold solutions, 
as described in \cite{MatPol.Dynamics,MurZho.Threshold}).

An initial condition $X_{\tau}$ ensures non-extinction if 
\beq
\tag{NEC}
\exists \alpha \in (\theta_c, 1], \, \exists \tau_0 \in \R, \, \forall \tau \in \R^d, \, \f{X_{\tau}}{X_{\tau} + N_0} \geq v_{\alpha}(\tau+ \tau_0),
\label{nonext}
\eeq
where $v_{\alpha}$ is the ``$\alpha$-bubble'' used in Theorem \ref{thm:invasion}.

\paragraph{An example.}
Now, we assume that we have a known number of mosquitoes to release, say $N$. 
When we release mosquitoes in the field (out of boxes), they will spread out to find vertebrates to feed on (if not fed in the lab prior to the release), to mate or to rest. 
Many environmental factors may influence their spread (see \cite{Mac.Influence}). 
As a very rough estimate we consider that the distribution of the released mosquitoes can be described by a Gaussian. A Gaussian profile is typically the result of a diffusion process.
However, we shall not use very fine properties of these profiles, and mainly focus on a ``significant spread radius'',
so that this assumption is not too restrictive.

Due to the above  simplification, the set of releases profiles (``RP'') for 
a total of $N$ mosquitoes at $k$ locations in a domain $[-L, L]^d$ is defined as
\begin{equation}
RP_k^d (N) := \Big\{ \tau \mapsto \f{N}{k}  \sum_{i=1}^k \displaystyle\f{e^{- \f{(\tau - \tau_i)^2}{2 \sigma_i}}}{(2 \pi \sigma_i)^{d/2}}, 
 \text{ with } \tau_1,\dots,\tau_k \text{ in } [-L, L]^d, \, \sigma_i \in [\sigma_0 - \epsilon, \sigma_0 + \epsilon] \Big\},
\label{eq:rpset}
\end{equation}
where $\sigma_0$ is an estimated diffusion coefficient and $\epsilon > 0$ represents the uncertainty on this parameter ($\sigma_i$ is the ``significant spread radius'').
In other words, for any $i$ between $1$ and $k$, the release profile is locally at the $i$-th release point a centered Gaussian with fixed amplitude $N/k$ and variance $\sigma_i$.

The basic requirement for a release profile is that $\int_{\R^d} X_{\tau} d \tau = N$. It is obviously satisfied for the elements in $RP_k^d(N)$.

We use uniform measure on $\big( [-L, L]^d \times [\sigma_0 - \epsilon, \sigma_0 + \epsilon] \big)^k$ to equip $RP_k^d (N)$ with a probability measure, denoted $\mathcal{M}$ 
in the following.

According to our estimate, the success probability satisfies
\begin{multline}
\Pro[ \textit{Non-extinction after releasing $N$ mosquitoes at $k$ locations }] 
\\ \geq \Pro[ X_{\tau} (\omega) \textit{ satisfies \eqref{nonext}} ],
\tag{SP}
\label{sucessprobability}
\end{multline}
where $X_{\tau} (\omega)$ is taken in $RP_k^d (N)$ according to the uniform probability measure.

\subsection{First result: relevance of under-estimating success}

Though it may seem naive, our under-estimation by radii given in Theorem \ref{thm:invasion} 
is not extremely bad, and this can be quantified in any dimension $d$.
Indeed, in any dimension we can prove convergence of our under-estimation in \eqref{sucessprobability} 
to $1$ as the number of releases goes to infinity, if we fix the number of mosquitoes per release.

More precisely, we define for a domain $\Omega\subset\R^d$,
\begin{multline}
P^d_k (N, \Omega) := \mathcal{M}\big \{ (x_i)_{1 \leq i \leq k}, \exists \alpha \in (\theta_c, 1), \exists x_0 \in \Omega, 
\\ x_0 + B_{R_{\alpha}} \subset \Omega \text{ and } \forall x \in x_0 + B_{R_{\alpha}}, \, \f{N}{k} \sum_{i=1}^k G_{\sigma,d} (x - x_i) \geq \alpha \big \},
\end{multline}
where $G_{\sigma,d} (y) = \f{1}{(2 \pi \sigma)^{d/2}} e^{-|y|^2/ 2 \sigma}$ and $B_{R_{\alpha}}=B_{R_{\alpha}} (0)$ is the ball of radius $R_{\alpha}$, centered at $0$.
Then, the probability of success of a random (in the sense of Section \ref{stochasticframework})
$k$-release of $N$ mosquitoes in the $d$-dimensional domain $\Omega$ is bigger than $P^d_k (N, \Omega)$, because of Theorem \ref{thm:invasion}.

Fixing the number of mosquitoes per release and letting the number of releases go to $\infty$ yields:
\begin{proposition}
 Let $1 > \alpha > \theta_c$, $N \geq N^* := (2 \pi \sigma)^{d/2} \f{\alpha}{1 - \alpha} N_0$ and $\Omega \subset \R^d$ a compact set containing a ball of radius $R_{\alpha}$. Then, 
 \beq
  P^d_k (k N, \Omega) \xrightarrow[k \to \infty]{} 1.
  \label{proba1anydim}
 \eeq
 \label{prop:coupon}
\end{proposition}
\begin{proof}
  There are two ingredients for the proof: First, we minimize a Gaussian at $x$ on a ball centered at $x$ by its value on the border of the ball.
  Second, if we pick uniformly an increasing number of balls with fixed radius and center in a compact domain, then their union covers almost-surely any given subset 
  (this second ingredient is connected with the well-known coupon collector's problem).
  Namely,
  $$
  \lVert y \rVert \leq \sqrt{2 \sigma \log(2)} \implies e^{- \lVert y \rVert^2 / 2 \sigma} \geq 1/2.
  $$
    
    Now, when we pick uniformly in a compact set the centers of balls of fixed radius $\alpha$, the probability of covering a given subset $\Omega_c \subset \Omega$ increases with the number $k$ of balls.
    Therefore it has a limit as $k \to + \infty$. In fact, this limit is equal to $1$.
    
    One can prove this claim using the coupon collector problem (see the classical work \cite{Erd.Classical} for the main results on this problem), after selecting a mesh for the compact domain $\Omega_c$.
    We take this mesh such that each cell has diameter less than $\sqrt{2 \sigma \log(2)}/2$, and positive measure.
    The domain $\Omega$ is compact, hence finitely many cells is enough.
    Picking the center of a random ball in a given cell of the mesh has probability $>0$, and we simply need to have picked one center in each element to be done.
    It remains to choose the (compact) set $\Omega_c = B_{R_{\alpha}} + x_0 \subset \Omega$ to conclude the proof.
\end{proof}

\begin{remark}
 We could have been a little more precise, and get an estimate for the expected value of the number $k$ of small balls required to cover the domain.
 According to classical results \cite{Erd.Classical} on the coupon collector problem, it typically grows as $N_c \log(N_c)$, where $N_c$ is the number of cells. 
    If the domain $\Omega$ has diameter $R$, $N_c$ is typically $(2 R/\sqrt{2 \sigma \log(2)})^{d}$, in dimension $d$.
    
    Therefore we should expect $\mathbb{E} [k] \sim d \big( \f{2 R}{\sqrt{2 \sigma \log(2)} }\big)^d \log( \f{2 R}{\sqrt{2 \sigma \log(2)}})$, and for a typical release area $R$ should be of the same order as $R_{\alpha}$.
   
\end{remark}

In fact, any $N > 0$ enjoys the same property, but then we need to assume that each cell contains a large enough number of release points.

\begin{corollary}
 For any $N > 0$ and $\alpha \in (\theta_c, 1)$, for $\Omega \subset \R^d$ a compact set containing a ball of radius $R_{\alpha}$, then for any compact subset $\Omega_c \subset \Omega$ containing a ball of radius $R_{\alpha}$ we have
 \[
  P_k^d (k N, \Omega_c) \xrightarrow[k \to \infty]{} 1.
 \]
\end{corollary}
\begin{proof}
 Let $\iota = \llceil \frac{N^*}{N} \rrceil$. With the same technique as for proving Proposition \ref{proba1anydim}, we get a coupon collector problem where $\iota$ coupons of each kind must be collected, whence the result.
\end{proof}

\subsection{Proof of invasiveness in Theorem \ref{thm:invasion} in any dimension}
\label{sub:thmproof}

We consider in this section the proof of Theorem \ref{thm:invasion} in any dimension. 
The case $d=1$ is postponed to the next section.

We use an approach based on the energy as proposed by \cite{MurZho.Threshold}.
In the present context, the energy is defined by
\begin{equation}\label{energy}
  E[u] = \int_{\R^d} \big( \f{\sigma}{2} \lvert \nabla u \rvert^2 - F(u(x))\big) dx.
\end{equation}
It is straightforward to see that if $p$ is a solution to \eqref{eq:p}, 
then the energy is non-increasing along a solution, \textit{i.e.},
$$
\frac{d}{dt} E[p] = - \int_{\R^d} \big( \sigma \Delta p + f(p)\big)^2 \,dx \leq 0.
$$
Thus, $E[p](t)\leq E[p^0]$ for all nonnegative $t$ and for $p$ solution to \eqref{eq:p}.
Moreover, Theorem 2 of \cite{MurZho.Threshold} states that if 
$\lim_{t\to +\infty} E[p(t,\cdot)] <0$, then $p(t,\cdot)\to 1$ locally
uniformly in $\R^d$ as $t\to +\infty$.
Thus, since $t\mapsto E[p(t,\cdot)]$ is non increasing, it is enough to 
choose $p^0$ such that $E[p^0]<0$ to conclude the proof of Theorem \ref{thm:invasion}.

For any $\alpha > \theta_c$, we construct $p^0(x)=v_\alpha(|x|)$ as defined in the statements 
of Theorem \ref{thm:invasion}.
To do so, consider the family of non-increasing radially symmetric functions, 
compactly supported in $B_{R_0}$ with $R_0>0$,
indexed by a small radius $0 < r_0 < R_0$,
defined by
$\phi(r) = 1$ if $r \leq r_0$, $\phi(r) = \f{R_0 - r}{R_0 - r_0}$ if $r_0 < r < R_0$, and $\phi(r) \equiv 0$ if $r > R_0$.

For any $0 < r_0 < R_0$ $\phi$ is continuous and piecewise linear. 
We define $v_\alpha(r)=\alpha \phi(r)$, for $r\geq 0$.
By the comparison principle, 
it suffices to find $(r_0, R_0)$ such that $E[\alpha \phi] < 0$ 
to ensure that $R_{\alpha} = R_0$ is suitable in Equation~\eqref{def:Ralpha} of 
Theorem \ref{thm:invasion}.
To do so, we introduce
\begin{equation}
J_d(r_0, R_0, \alpha, \phi) := \f{E[\alpha \phi]}{\lvert S^{d-1} \rvert} = \alpha^2 \sigma \int_0^{\infty} r^{d-1} \lvert \nabla \phi (r) \rvert^2 dr 
- \big( \f{r_0^d}{d} F(\alpha) + \int_{r_0}^{R_0} r^{d-1} \int_0^{\alpha \phi(r)} f(s) ds dr \big).
\end{equation}

Now, we use our specific choice of non-increasing radially symmetric function $\phi$.
Introducing $\rho := r_0 / R_0$, and with obvious abuses of notation, $J_d$ stands again for
\begin{equation}
 J_d (\rho, R_0, \alpha) := R_0^d \Big( \f{\sigma}{d R_0^2} \f{1 - \rho^d}{(1 - \rho)^2} - F(\alpha) \f{\rho^d}{d} 
  - \f{1-\rho}{\alpha}\int_0^{\alpha} \big(1 - \f{1-\rho}{\alpha} x \big)^{d-1} F(x) dx \Big),
\end{equation}
where $F$ is the antiderivative of $f$ (as introduced in \eqref{def:F}).
After an integration by parts, we have
$$
J_d (\rho, R_0, \alpha) = R_0^d \Big( \f{\sigma}{d R_0^2} \f{1 - \rho^d}{(1 - \rho)^2} -
\int_0^{\alpha} \big(1 - \f{1-\rho}{\alpha} x \big)^d f(x) dx \Big).
$$
We choose $\rho\in (0,1)$ such that 
\begin{equation}
\int_0^{\alpha} \big(1 - \f{1-\rho}{\alpha} x \big)^d f(x) dx >0
  \label{positiveR}
\end{equation}
Then the energy $J_d(\rho,R_0,\alpha)$ decreases to $- \infty$ with $R_0$ and is positive for $R_0 \to 0$, 
so the minimal scaling ensuring negative energy is obtained for some known value of $R_0 =: R^{(d)}_{\alpha} (\rho)$, such that $J_d (\rho, R^{(d)}_{\alpha} (\rho), \alpha) = 0$.
Namely,
\begin{equation}
 \big( R_{\alpha}^{(d)} (\rho) \big)^2 = \sigma \f{1 - \rho^d}{(1 - \rho)^2} \f{1}{\int_0^{\alpha} \big( 1 - \f{1-\rho}{\alpha} x \big)^d f(x) dx},
 \label{Rsquare}
\end{equation}
which is a rational fraction in $\rho$.
Thus we recover formula \eqref{def:Ralpha} in Theorem \ref{thm:invasion} by minimizing
with respect to those $\rho$ satisfying constraint \eqref{positiveR}.
\qed

We examine in particular formula \eqref{Rsquare} in the case $d=1$.
To do so, we introduce 
\begin{equation}
 U(\alpha) := F(\alpha) - \f{1}{\alpha} \int_0^{\alpha} F(x) dx, \quad V(\alpha) := \f{1}{\alpha} \int_0^{\alpha} F(x) dx.
 \label{eq:UV}
\end{equation}
Since $F(x) \leq F(\alpha)$ for $x \leq \alpha$, we know that $U$ is positive and $V$ is increasing  with respect to $\alpha$ ($V'(\alpha) = \f{1}{\alpha} U(\alpha)$). Moreover, $V(\theta_c) < 0$.
We get
\begin{equation}
 R^{(1)}_{\alpha} (\rho) = \f{\alpha \sqrt{\sigma}}{\sqrt{(1-\rho) (V(\alpha) + \rho U(\alpha) )}},
\end{equation}
under the constraint $V(\alpha) + \rho U(\alpha) > 0$.
The optimal choice for $\rho$ is then $\rho^*_1 (\alpha) := \f{1}{2} - \f{1}{2} \f{V(\alpha)}{U(\alpha)}$.
It satisfies $V(\alpha) + \rho^*_1 (\alpha) U(\alpha) > 0$ since $U(\alpha) = F(\alpha) - V(\alpha) > 0$ and $F(\alpha) > 0$.

Finally, $\rho^*_1$ corresponds to a minimal radius
\begin{equation}
 R^{(1), *}_{\alpha} := R^{(1)}_{\alpha} (\rho^*_1(\alpha)) = 2 \sqrt{\sigma} \f{\alpha \sqrt{U(\alpha)}}{F(\alpha)},
\end{equation}
with $U(\alpha)$ as in \eqref{eq:UV}.

\begin{remark}
We emphasize that $R_\alpha$ 
quantifies the minimal radius which ensures invasion from level $\alpha$,
 in the sense that it provides an upper bound for it.
However, we were not able to perform an analytical computation
of the actual optimal radius (=support size) of a critical bubble.
\end{remark}

\begin{remark}
 We note in passing that the same energy \eqref{energy} appears for instance in the review paper \cite{BarHew} and in associated literature, but is used in a different spirit (stemming from statistical physics).
\end{remark}

Before restricting to dimension $1$ in the sequel, we end the general exposition in this section with a numerical illustration. In order to help the reader getting a clearer picture of the invasion problem we investigate in the present paper, Figure~\ref{fig:evolution} displays the time dynamics of equation \eqref{AllenCahn} in two spatial dimensions, with three different initial conditions. It illustrates the fact that with a fixed number of release points taken uniformly in a rectangle, invasion typically appears only if the size of the rectangle is well chosen.

If it is too small (Figure \ref{fig:evolution}-Right) the pressure of the surrounding {\it Wolbachia}-free environment is too strong for the infection to propagate. If it is too large (Figure \ref{fig:evolution}-Left), the release points are likely to be too scattered and never reach and invasion threshold. Whereas in Figure \ref{fig:evolution}-Center, the release area and the number of releases is sufficient to generate a wide enough domain of {\it Wolbachia}-infected mosquitoes which spreads for larger times.

\begin{figure}
\includegraphics[width=.33\textwidth]{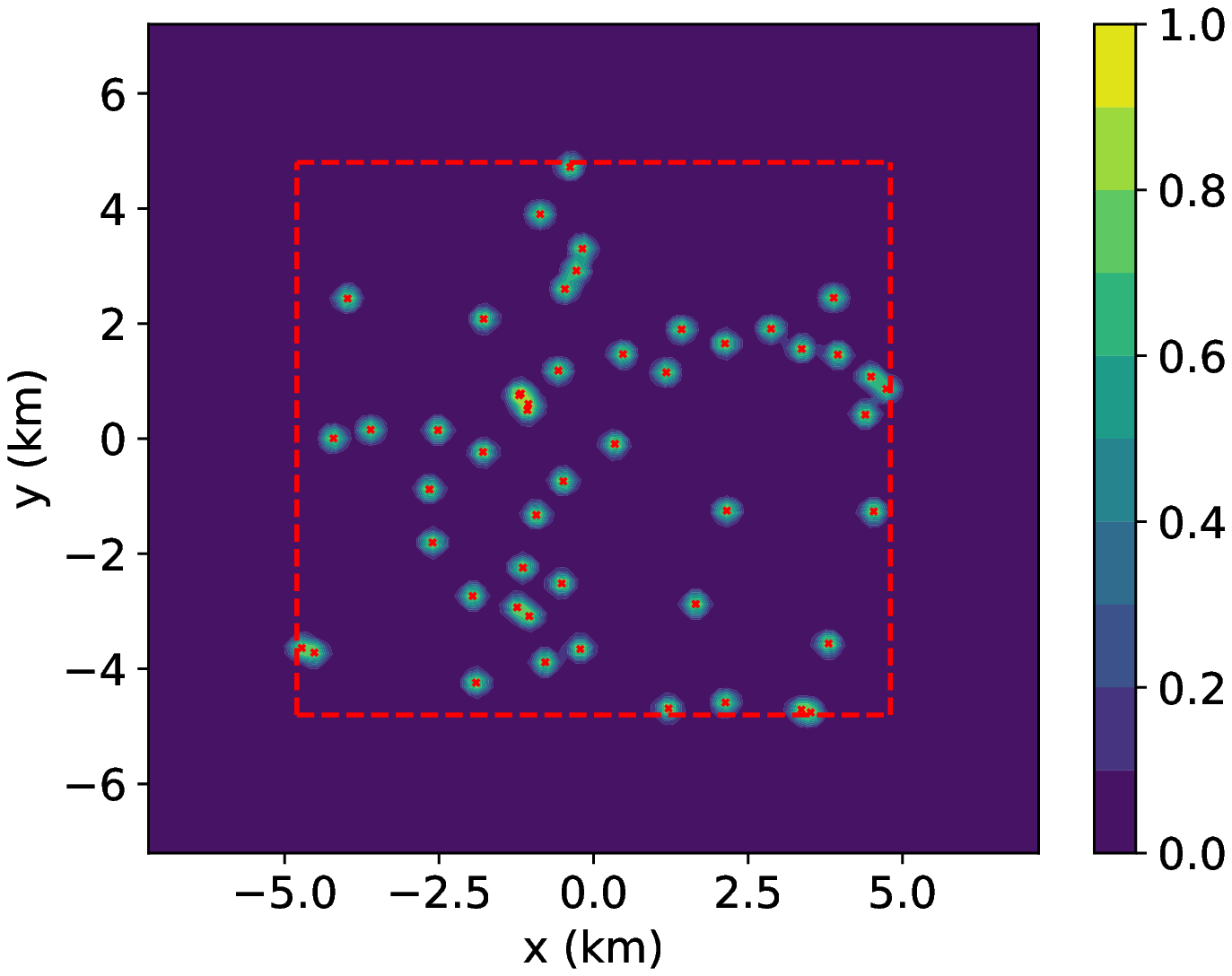}
\includegraphics[width=.33\textwidth]{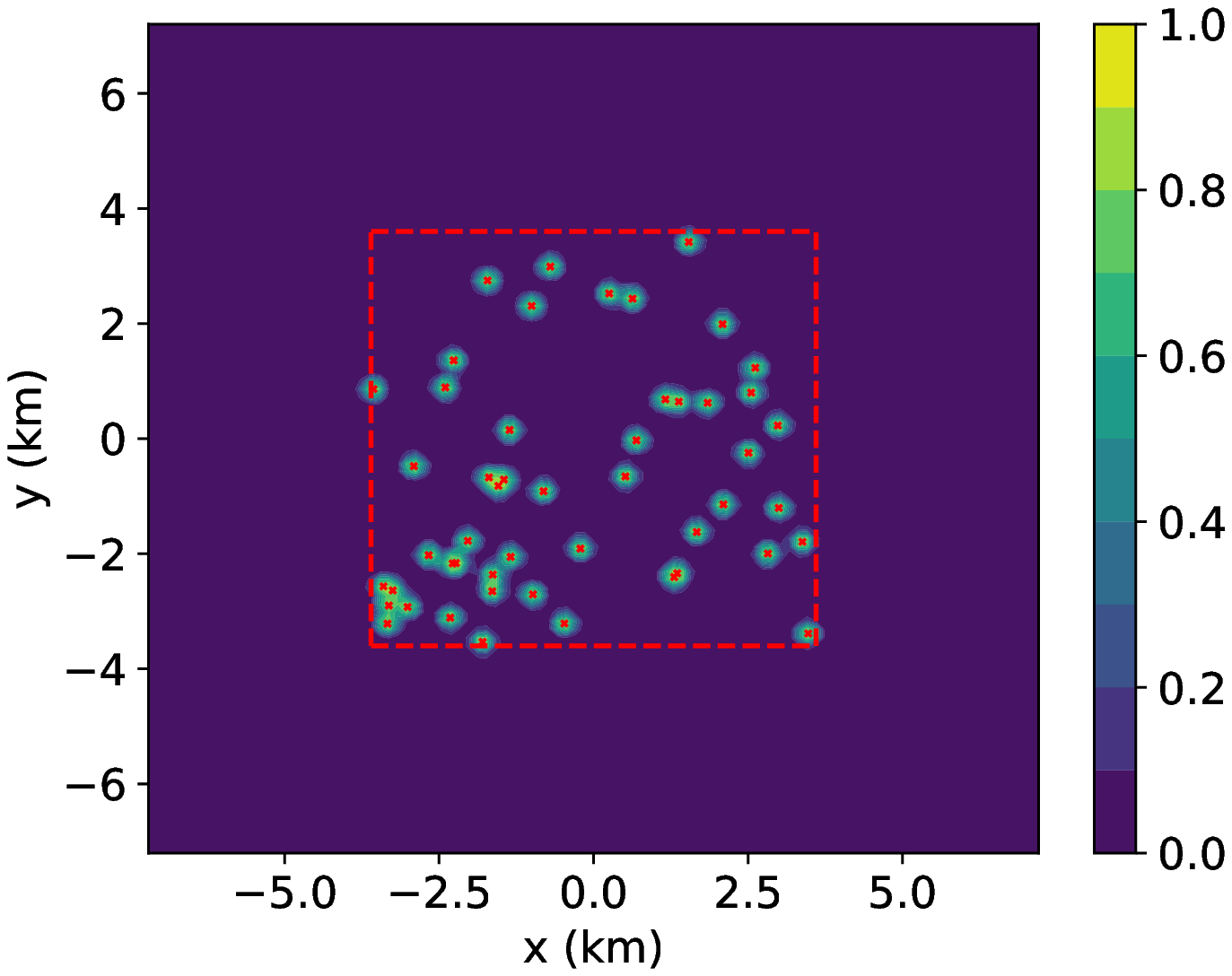}
\includegraphics[width=.33\textwidth]{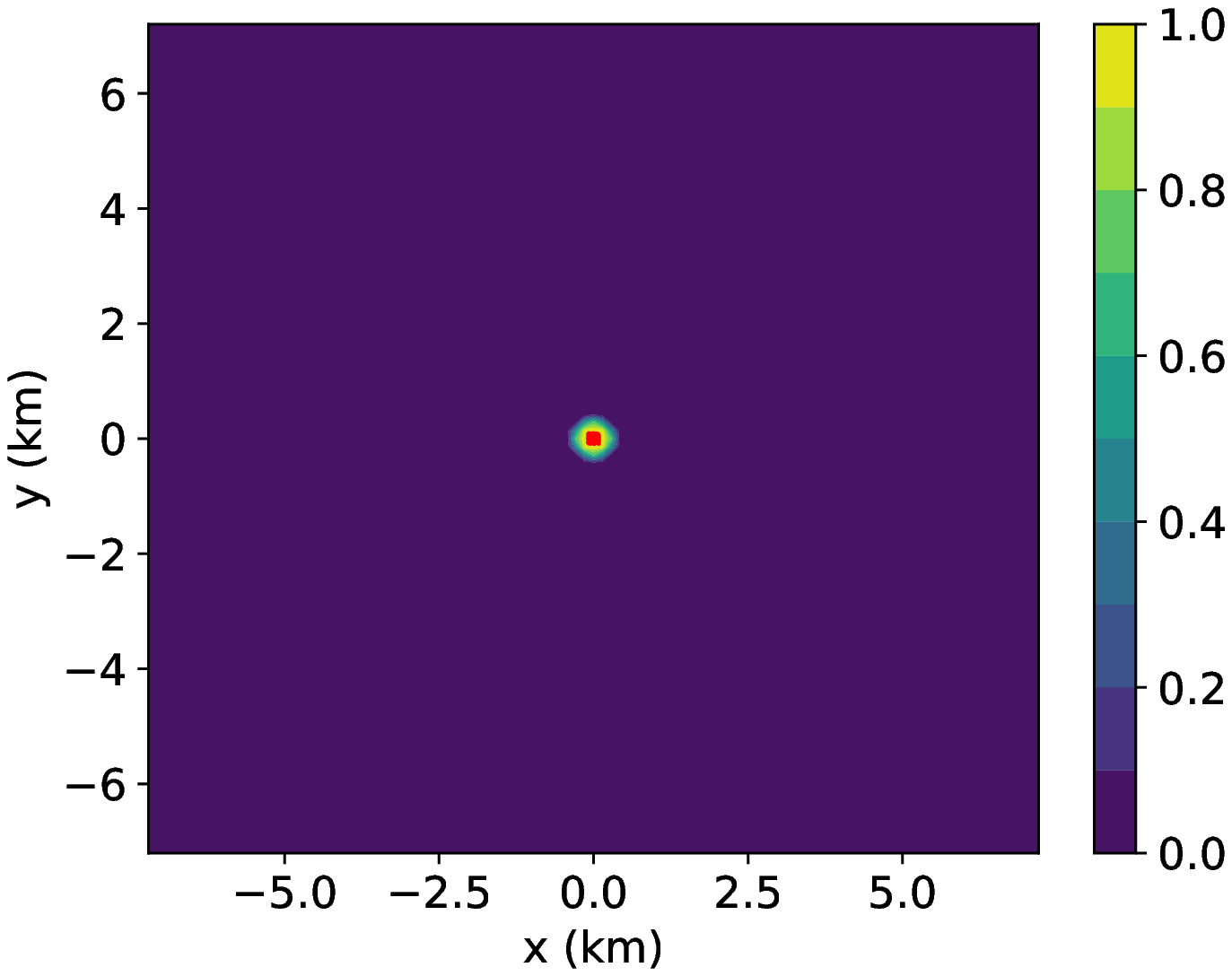}
\includegraphics[width=.33\textwidth]{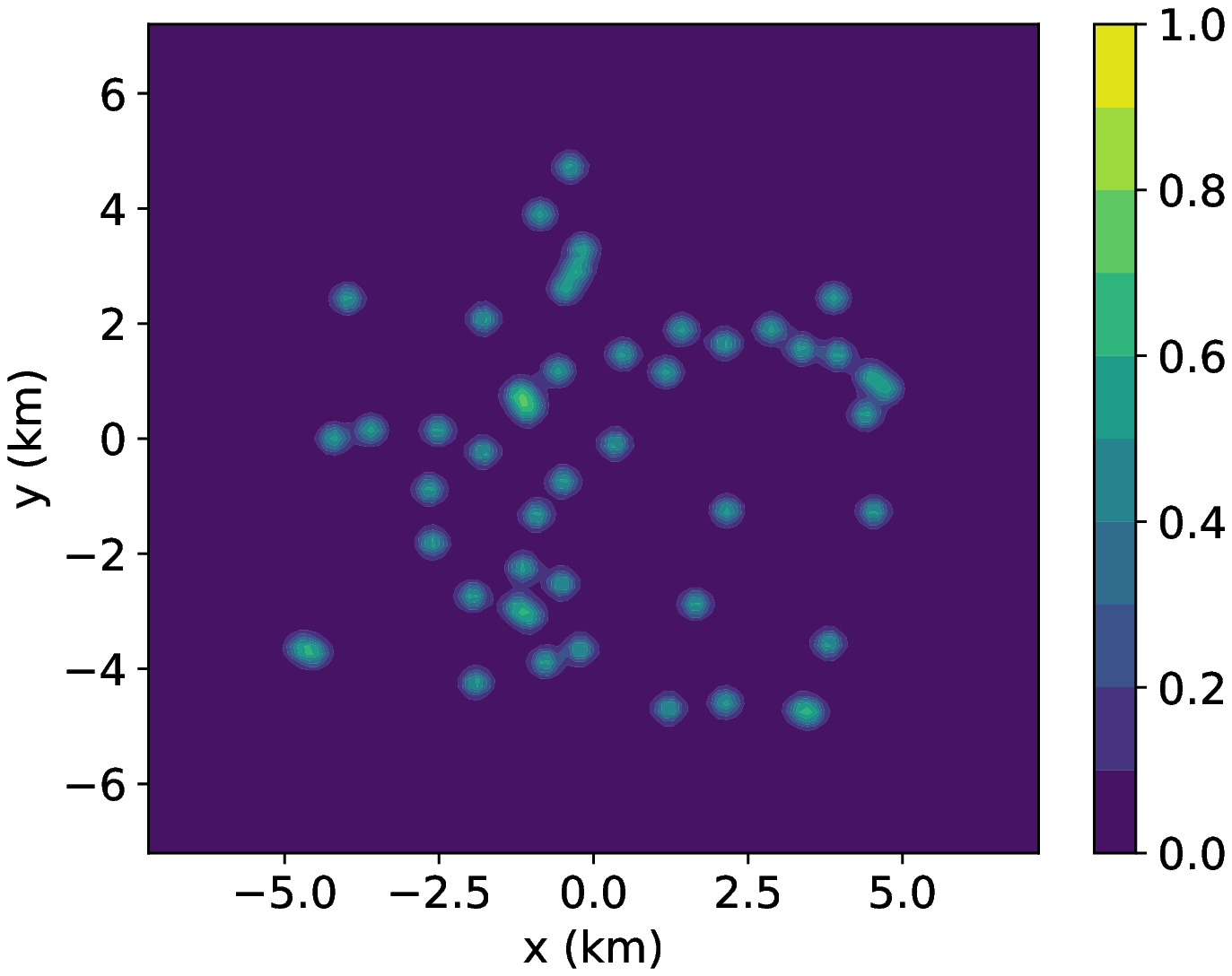}
\includegraphics[width=.33\textwidth]{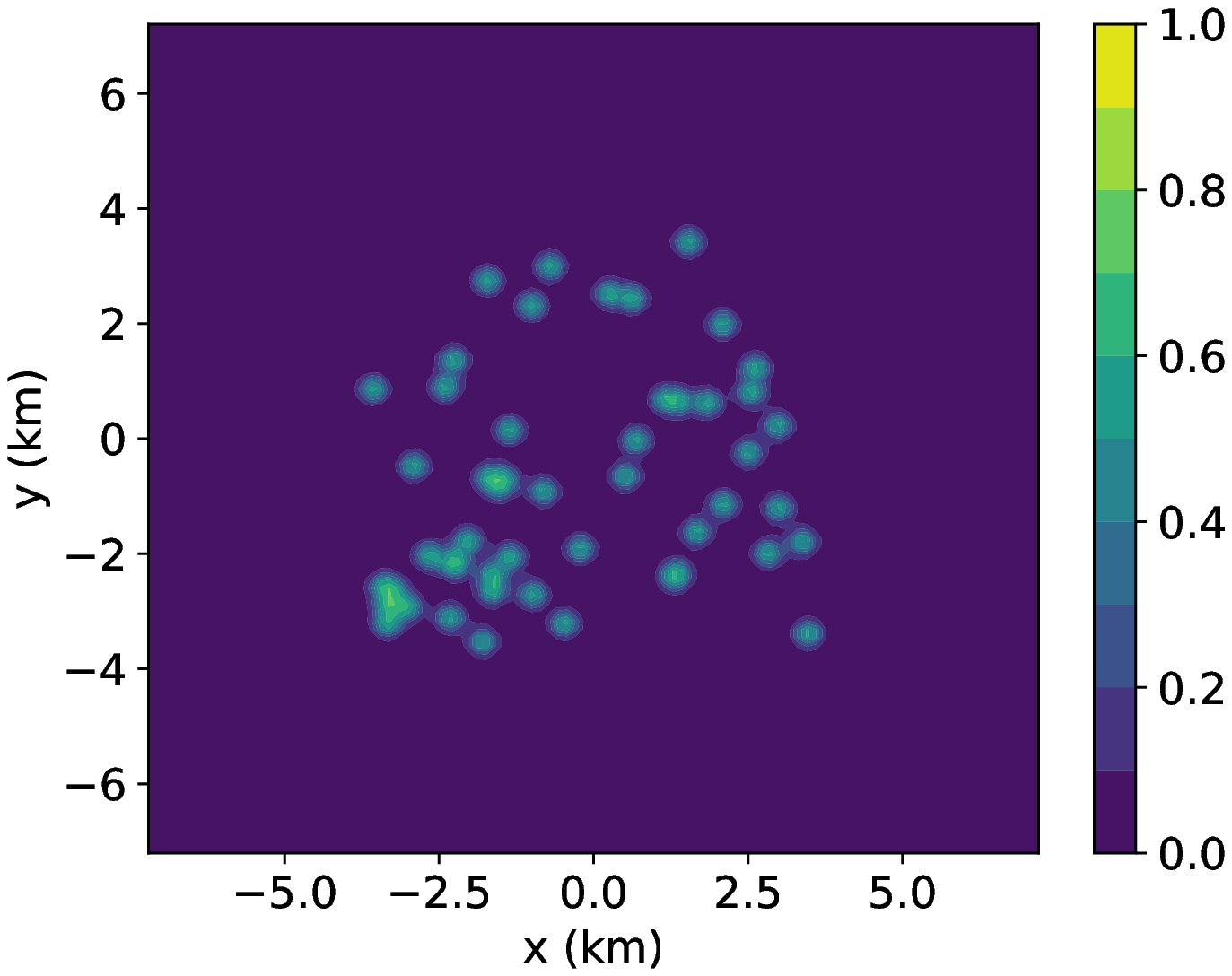}
\includegraphics[width=.33\textwidth]{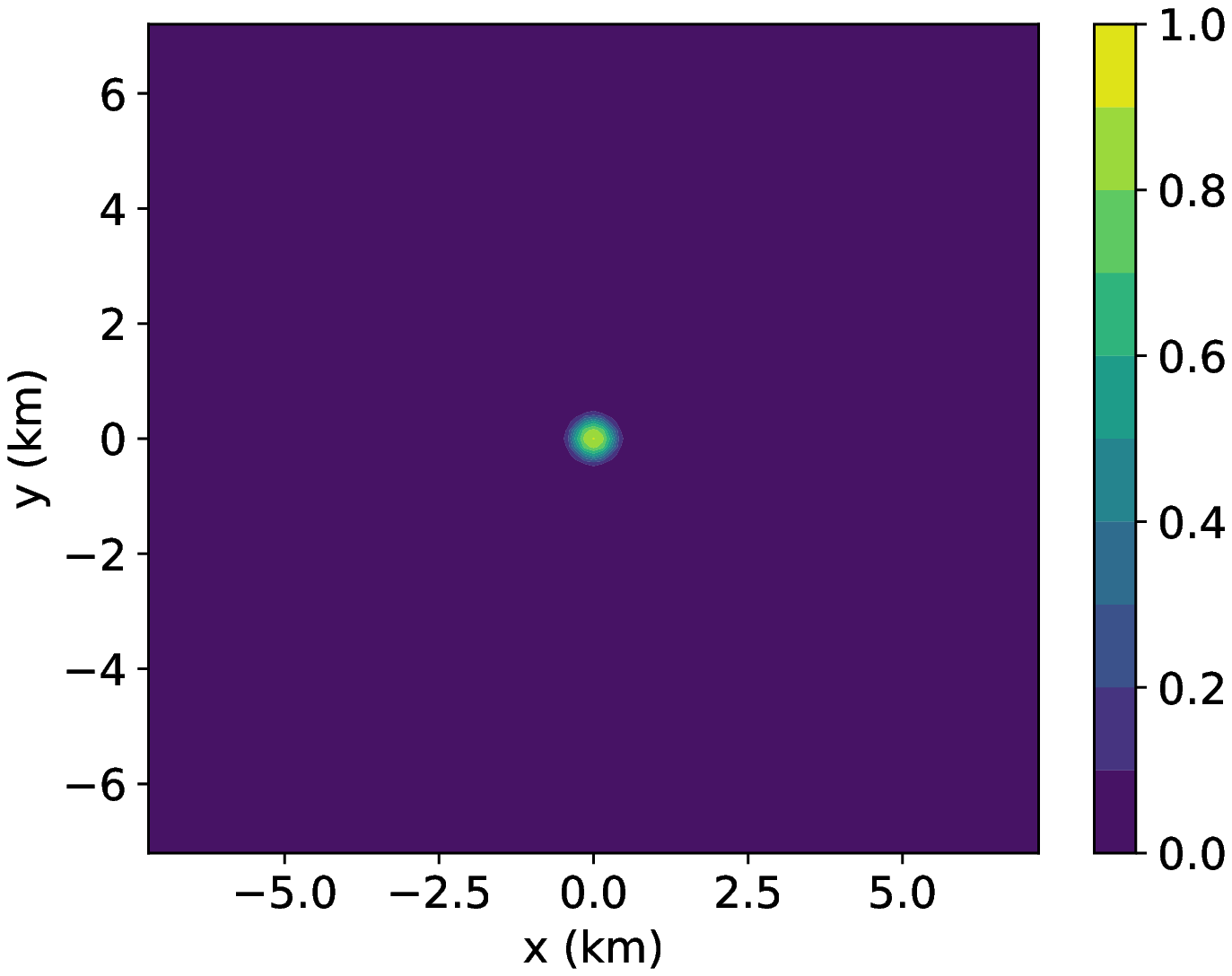}
\includegraphics[width=.33\textwidth]{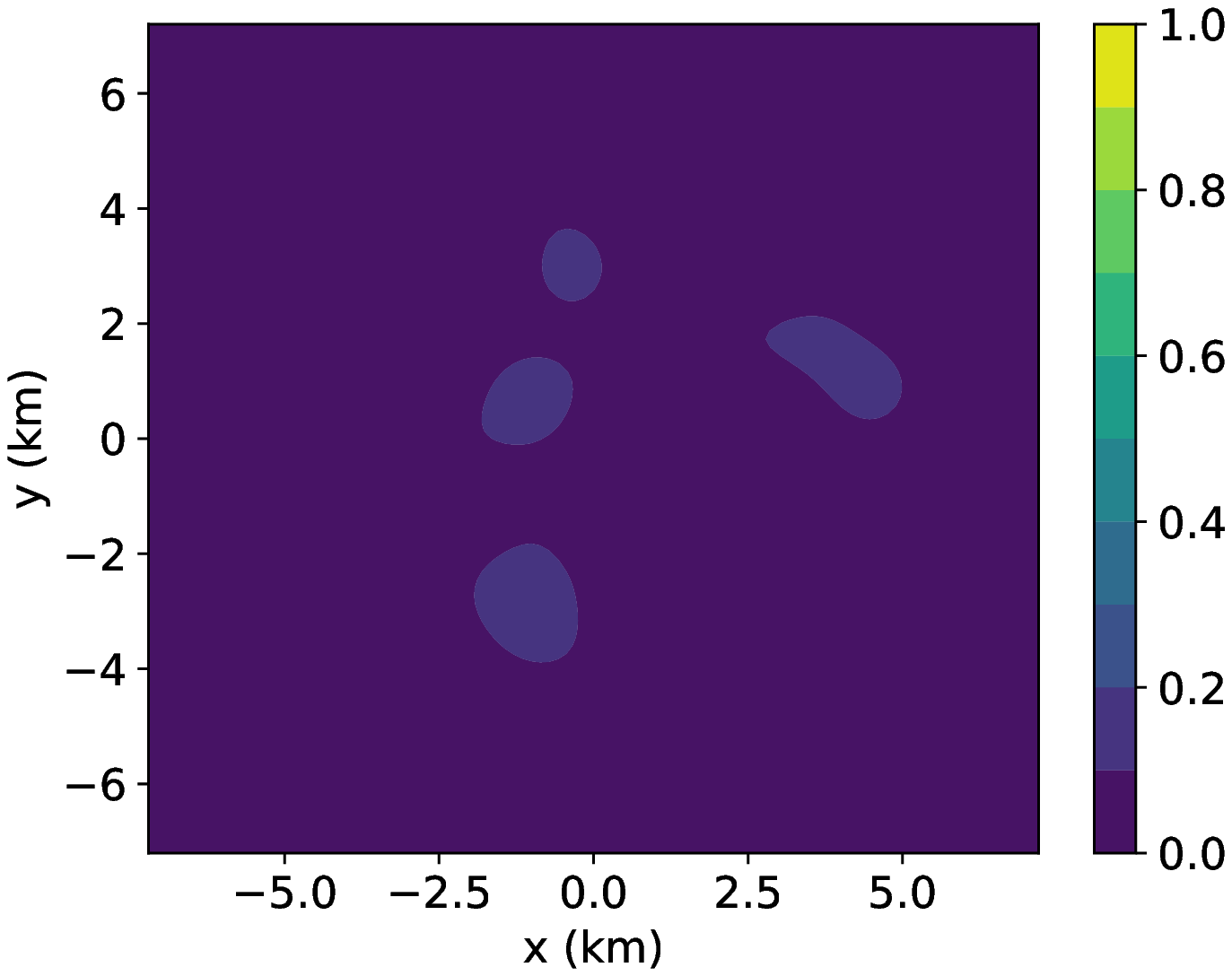}
\includegraphics[width=.33\textwidth]{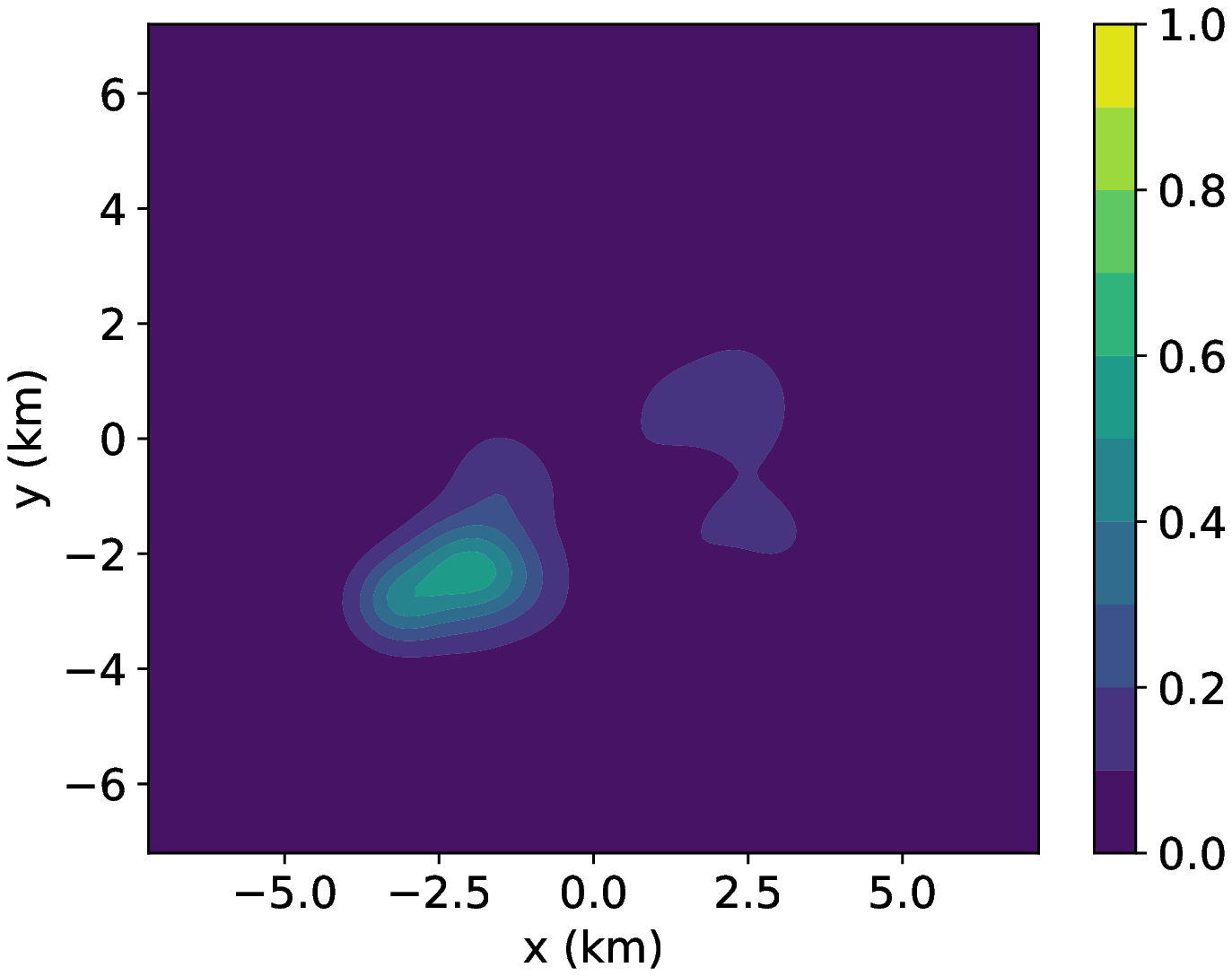}
\includegraphics[width=.33\textwidth]{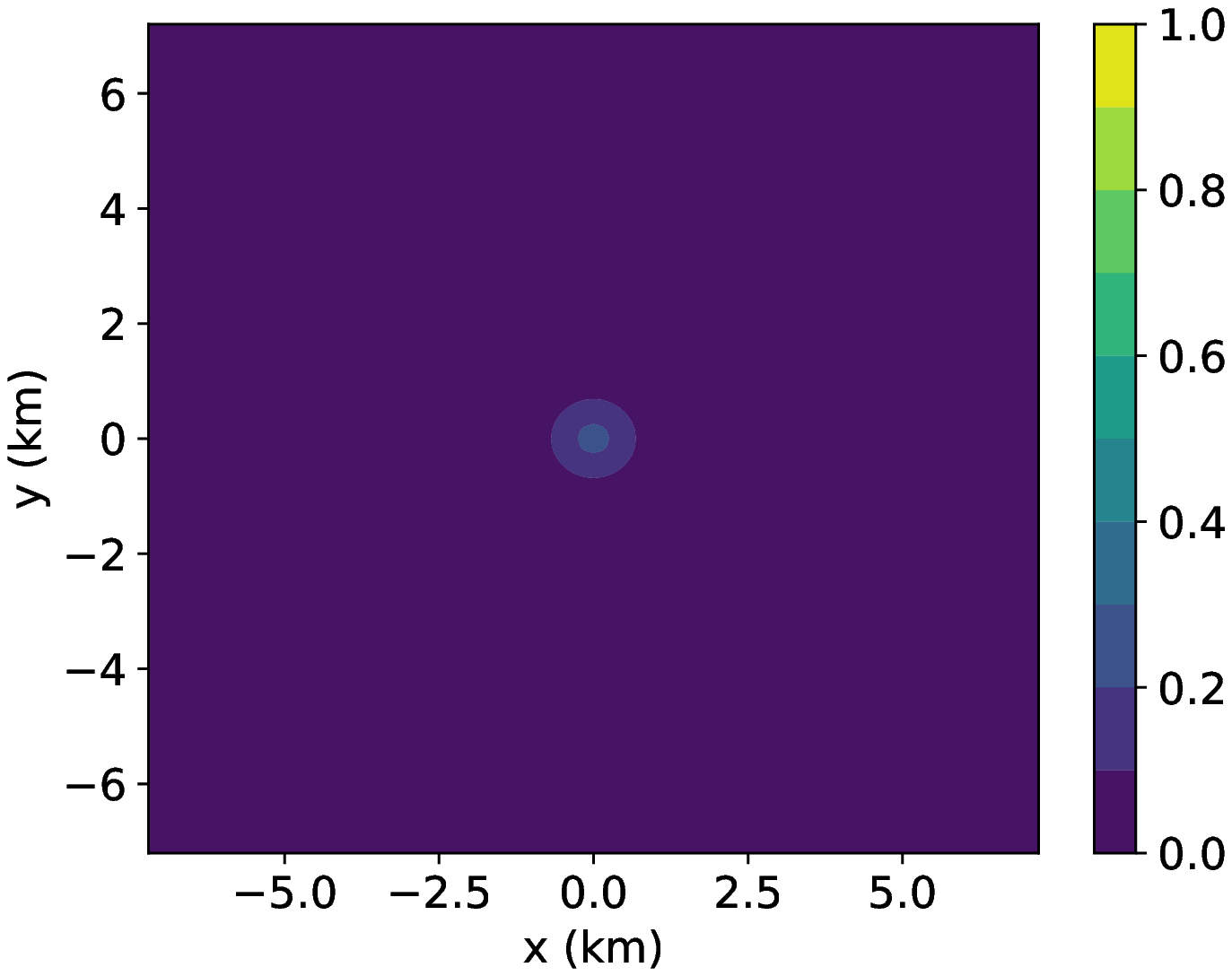}
\includegraphics[width=.33\textwidth]{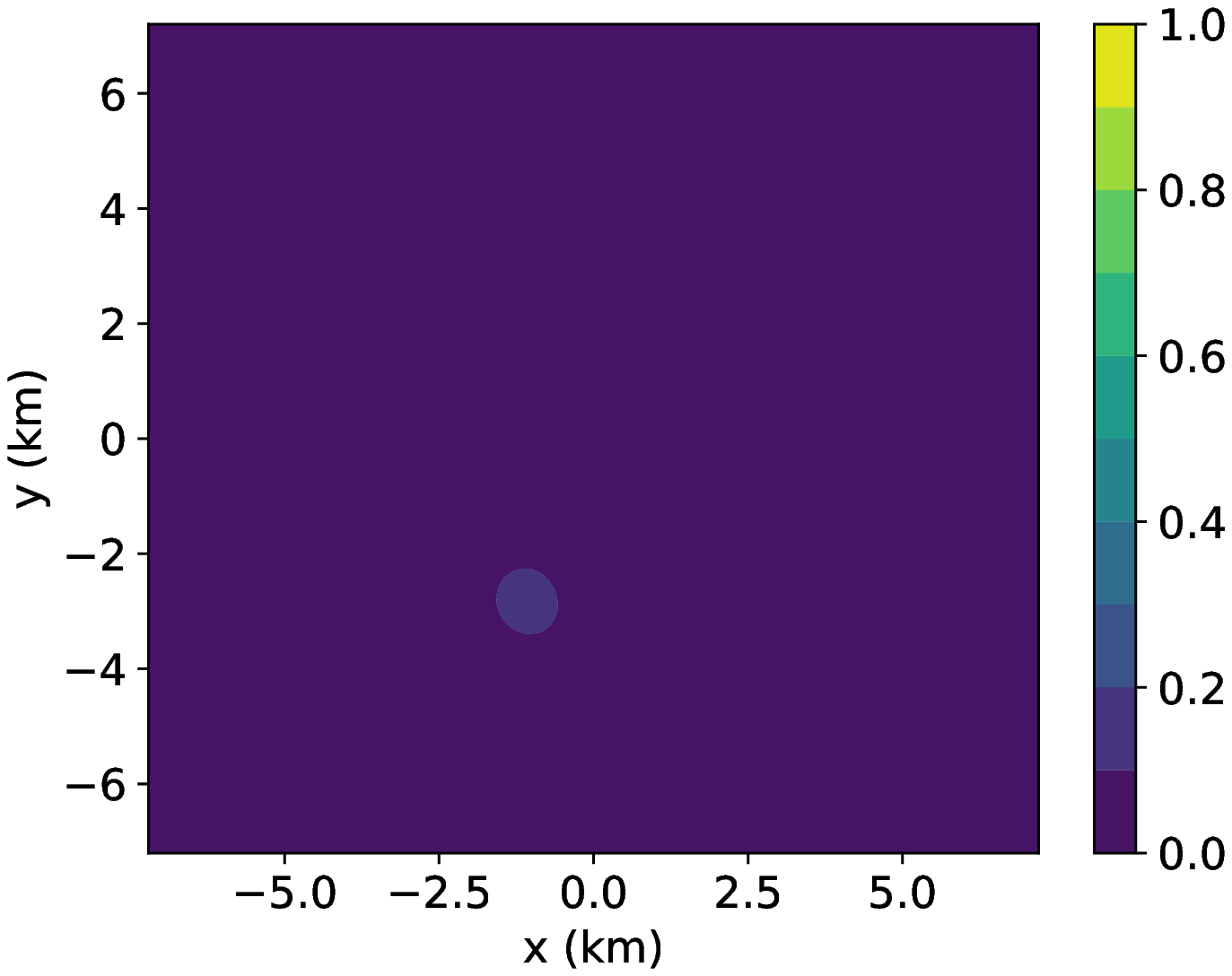}
\includegraphics[width=.33\textwidth]{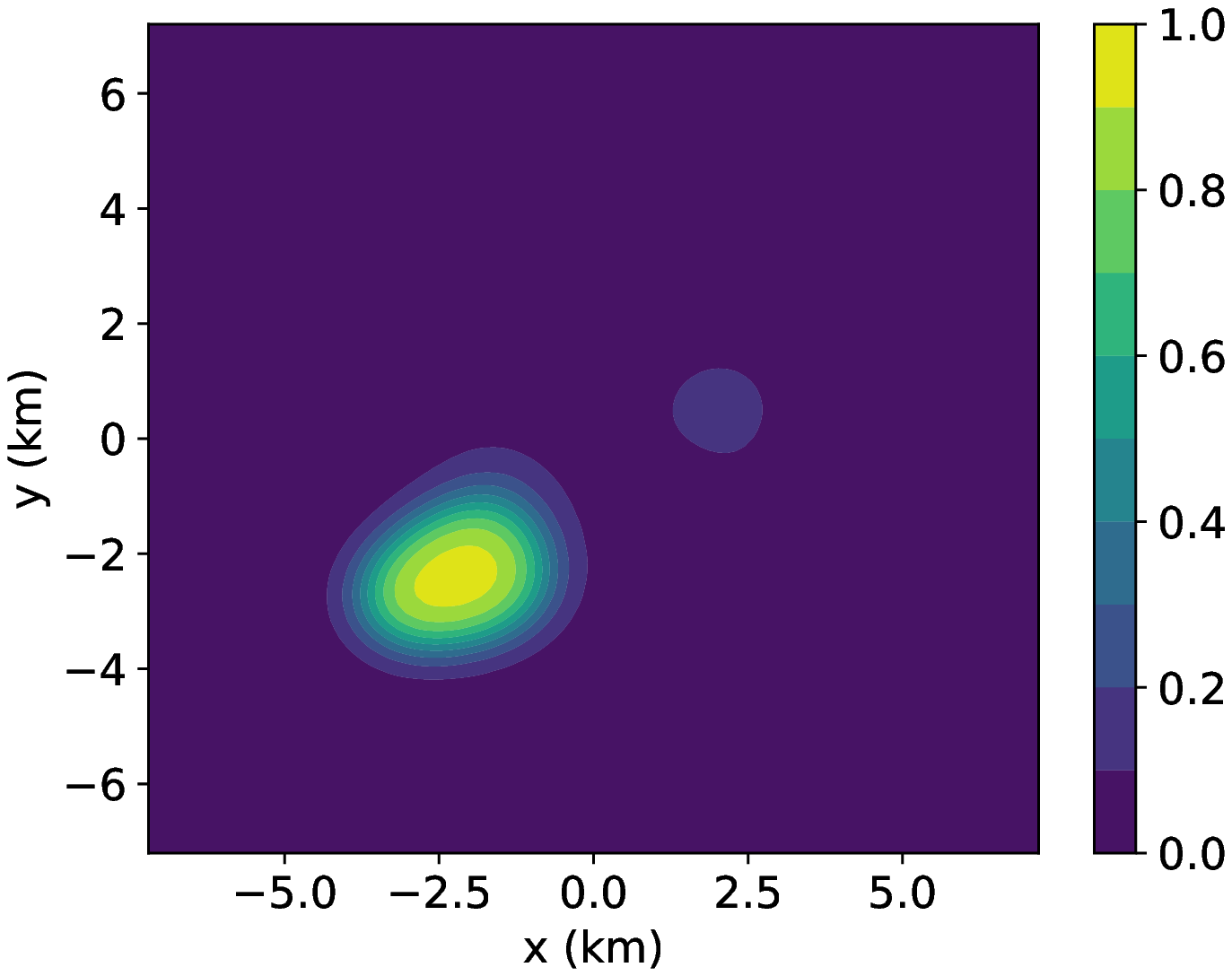}
\includegraphics[width=.33\textwidth]{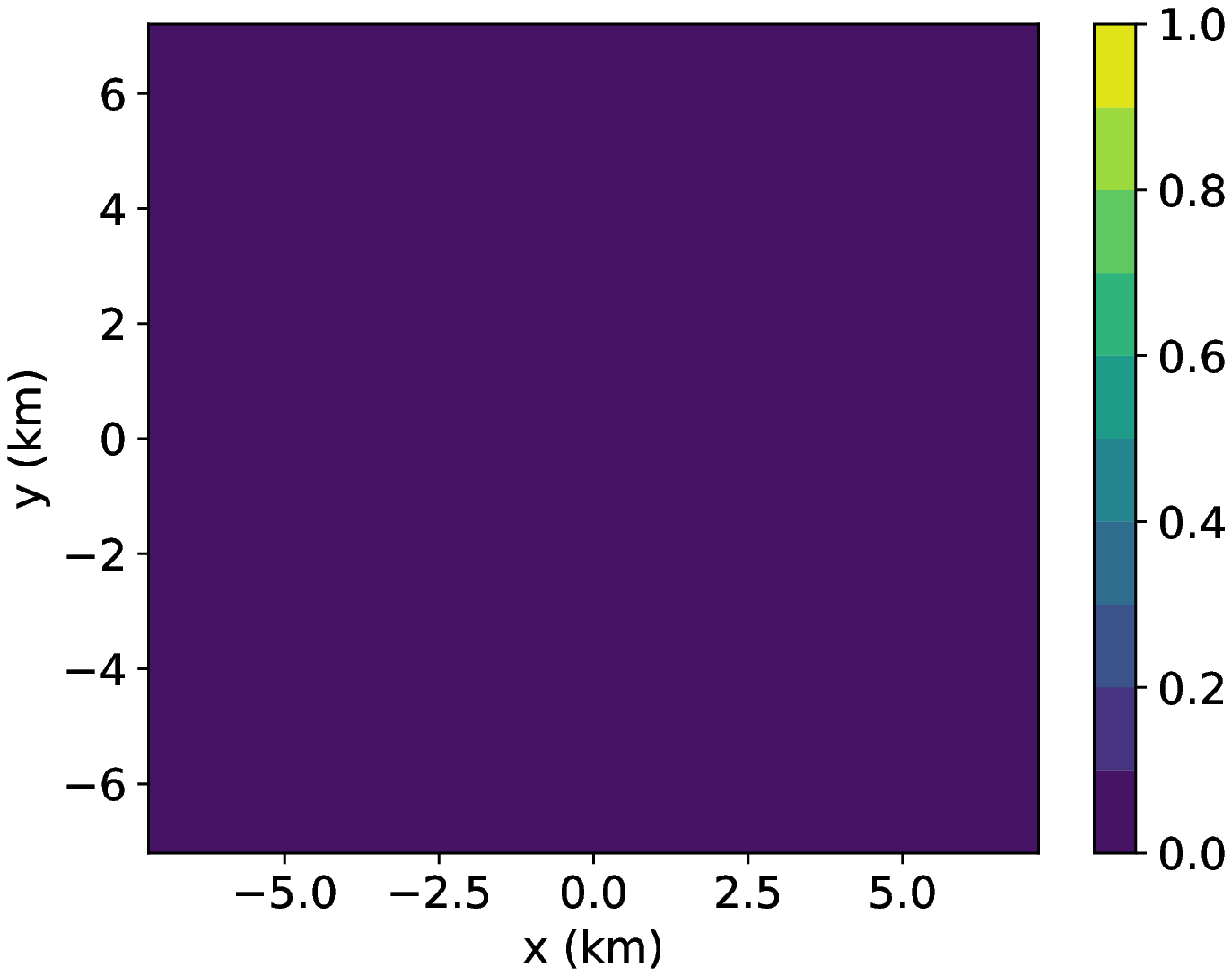}
\includegraphics[width=.33\textwidth]{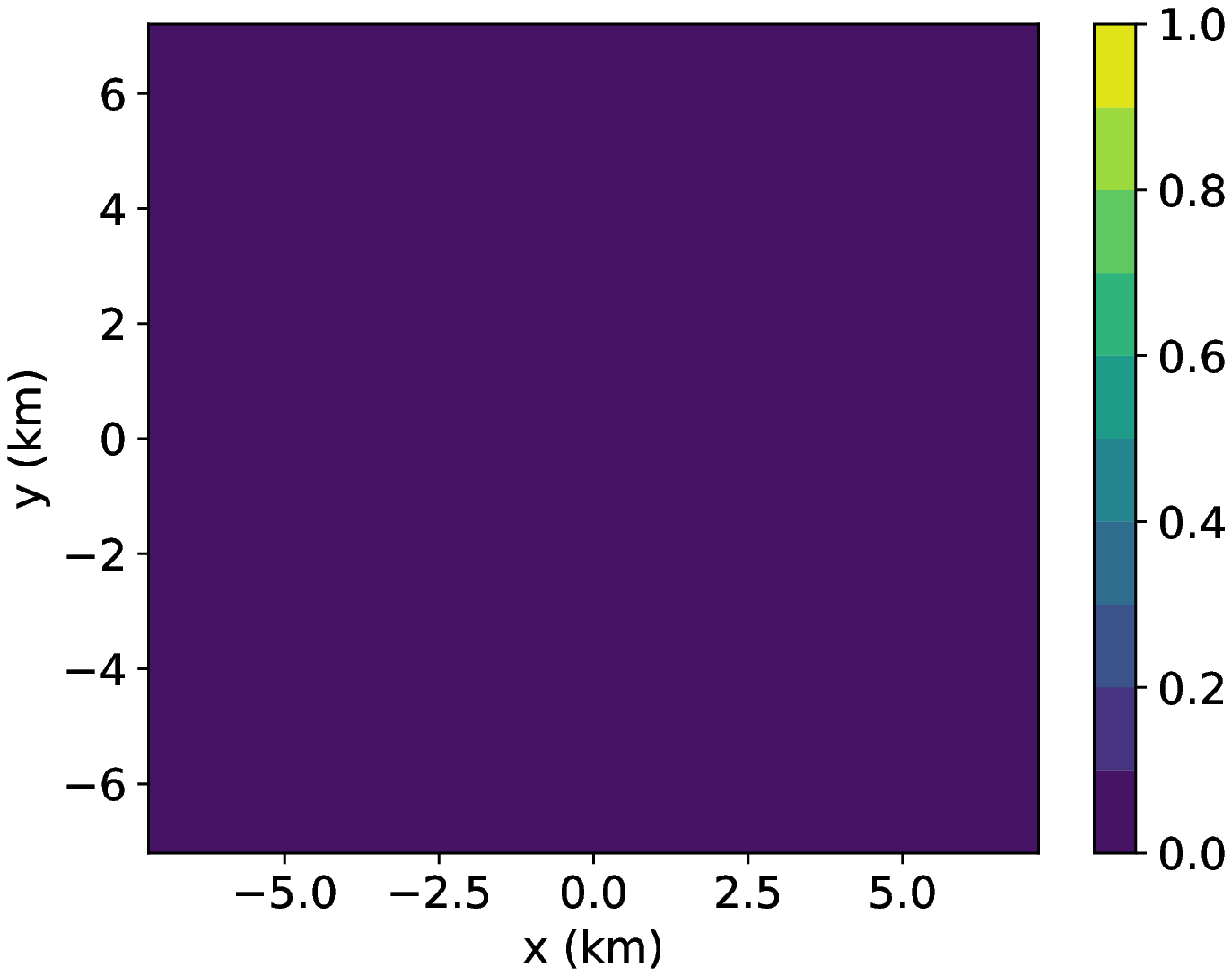}
\includegraphics[width=.33\textwidth]{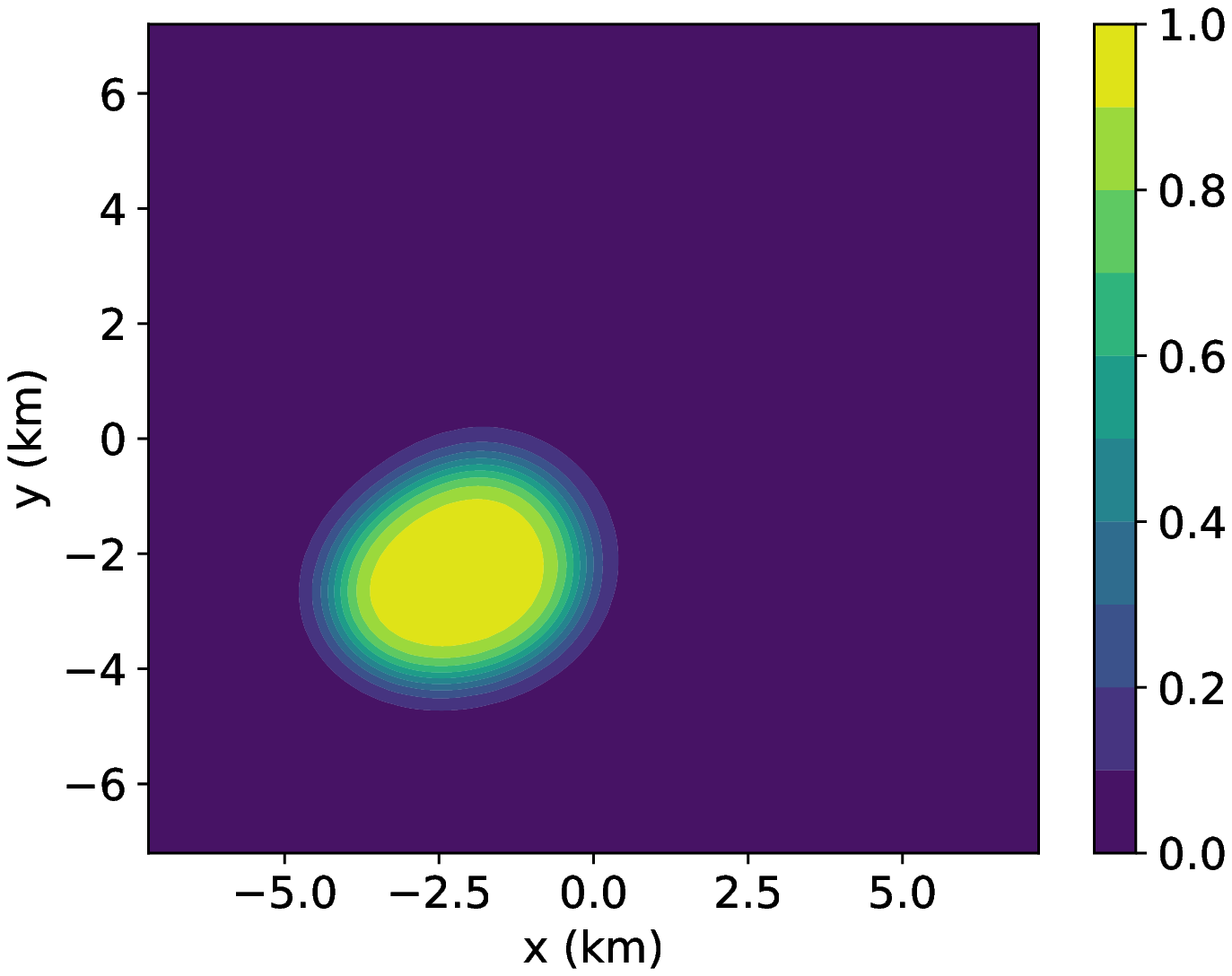}
\includegraphics[width=.33\textwidth]{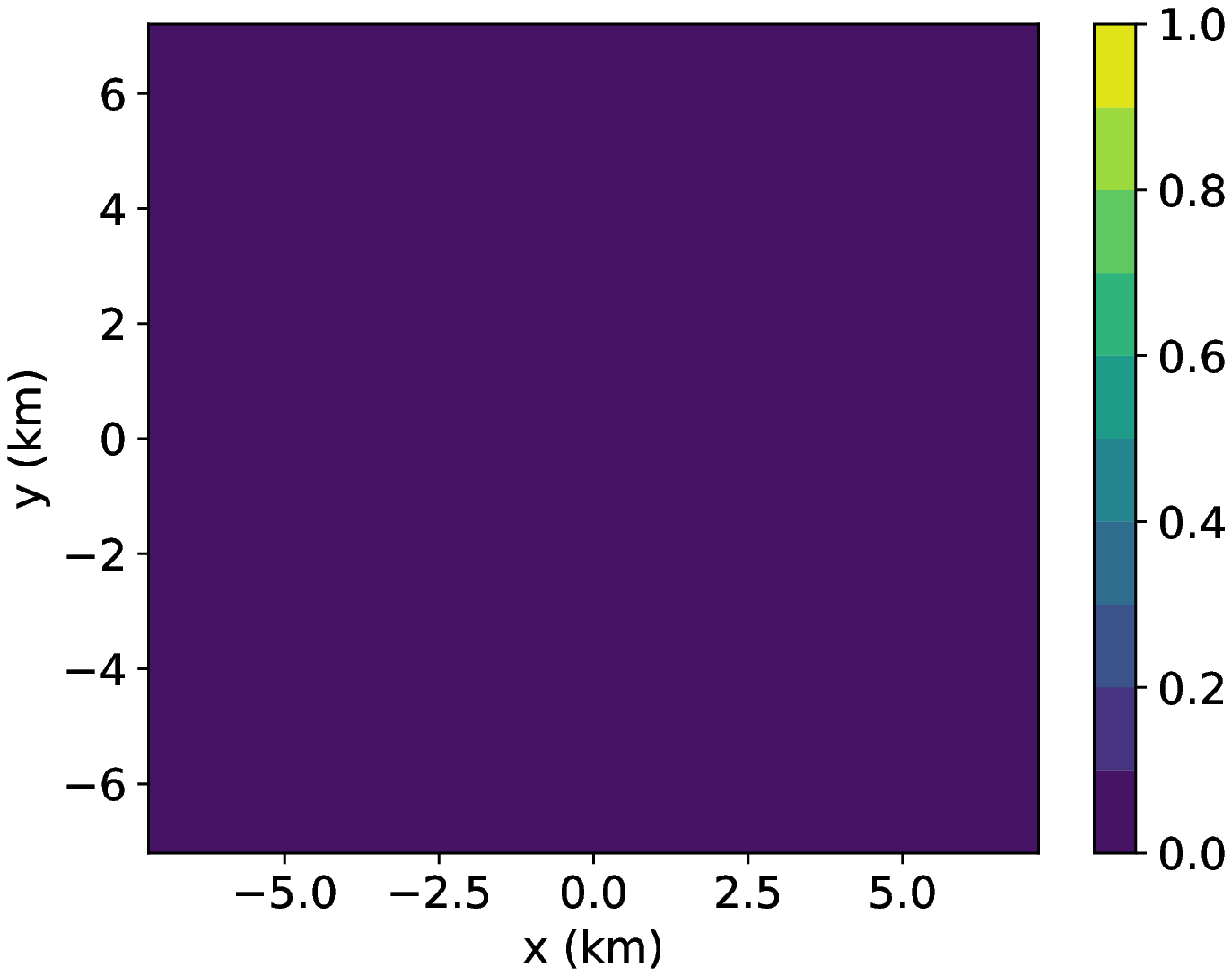}

\caption{Time dynamics with three different initial releases belonging to the set $RP_{50}^2(N)$ of \eqref{eq:rpset}, with $N/(N+N_0) = 0.75$. Integration is performed on the domain $[-L, L]$ with $L = 50 \textrm{km}$. The release box is plotted in dashed red on the first picture of each configuration. {\it Left}: Release box $[-2 L/3, 2 L/3]^2$. {\it Center}: Release box $[-L/2, L/2]^2$. {\it Right}: Release box $[-L/12.5, L/12.5]^2$. {\it From top to bottom}: increasing time $t \in \{0, 1, 25, 50, 75\}$, in days. The color indicates the value of $p$ (with the scale on the right).}
\label{fig:evolution}
\end{figure}

\section{Critical bubbles of non-extinction in dimension 1}
\label{sec:1D}

\subsection{Construction}

In this section, we consider the particular one dimensional case for which
we can construct a sharp critical bubble. 
To do so, we consider the following differential system:
\begin{equation}
\label{Palpha}
 \sigma u_{\alpha}'' + f(u_{\alpha}) = 0 \text{ in } \R_+,\quad  u_{\alpha}(0) = \alpha,\ u_{\alpha}'(0) = 0.
\end{equation}

\begin{proposition}\label{prop:ualpha}
System \eqref{Palpha} admits a unique maximal solution $u_{\alpha}$; 
it is global and can be extended by symmetry on $\R$ as a function of class $\mathcal{C}^2$. 
Moreover, if $\alpha > \theta_c$, then $L_\alpha$ defined in \eqref{def:lalpha} is finite and
$u_\alpha$ is monotonically decreasing on $\R_+$ and vanishes at $L_\alpha$.
\end{proposition}

\begin{definition}\label{def:valpha}
For $\alpha\in (\theta_c,1]$, we denote by an $\alpha$-bubble in one dimension the function $v_\alpha$
defined by $$v_\alpha(x)=u_\alpha(|x|)^+:=\max\{0, u_\alpha(|x|) \} \mbox{ .}$$
\end{definition}

From Proposition~\ref{prop:ualpha} and Definition~\ref{def:valpha} we have that
$v_\alpha$ is 
compactly supported with supp$(v_\alpha)=[-L_\alpha,L_\alpha]$.

\begin{proof}
Local existence is granted by Cauchy-Lipschitz theorem.
Then, we multiply Equation~\eqref{Palpha} by $u'_\alpha$,
\[
\frac{\sigma}{2} \big(  (u_{\alpha}')^2 \big)' + \big( F(u_{\alpha}) \big)' = 0,
\]
which implies (since $u_{\alpha}' (0) = 0, u_{\alpha}(0) = \alpha$ and the domain is connected) that:
\[
\frac{\sigma}{2} ( u_{\alpha}' )^2 = F(\alpha) - F(u_{\alpha}).
\]
Recall that $F(x) = \int_0^x f(y)dy$ is positive increasing from $\theta_c$. 
Hence, for $\alpha > \theta_c$, $u_{\alpha}$ stays strictly below $\alpha$ except at $0$; $u_{\alpha}'$ cannot vanish unless $u_{\alpha}= \alpha$. 
Hence, $u_{\alpha}$ is decreasing on $\R_+$.

Because $u_{\alpha}$ is decreasing, its derivative is negative and thus:
\begin{equation}\label{dualpha}
\sqrt{\sigma} \displaystyle\frac{du_{\alpha}}{dx} =  - \sqrt{2 (F (\alpha) - F(u_{\alpha}) )}.
\end{equation}
Then, $u_{\alpha}$, being monotone, is invertible on its range. Let us define $\chi_{\alpha}(u_{\alpha}(x))=x$, so that $u_{\alpha}(\chi_{\alpha}(\omega))=\omega$.
By the chain rule, we have
\[
\displaystyle\frac{d \chi_{\alpha}}{d \omega} = - \sqrt{\frac{\sigma}{2 (F (\alpha) - F(\omega) )}},
\]
so that,
\begin{equation}\label{chialpha}
\chi_{\alpha}(\omega) = \int_{\omega}^\alpha \sqrt{\frac{\sigma}{2 (F (\alpha) - F(v) )}}dv.
\end{equation}
Thus the function $\chi_{\alpha}$ evaluated at $\omega$ is equal to the unique radius at which the solution of \eqref{Palpha} takes the value $\omega$. 

Moreover, if $\alpha > \theta_c$, $F(\alpha) - F(v)$ vanishes if and only if $v=\alpha$. Therefore, if $v = \alpha - h$,
we can write $F (\alpha - h ) = F(\alpha) - h f(\alpha) + O (h^2)$, which means that locally:
\[
\frac{1}{\sqrt{ F (\alpha) - F(v)} } \underset{v \to \alpha}{\sim} \frac{1}{\sqrt{f(\alpha)}}\frac{1}{\sqrt{\alpha - v}},
\]
which is integrable as long as $f(\alpha) \not= 0$, which is true since $\alpha \in (\theta_c, 1) \subset (\theta, 1)$.

On the other hand, the integral diverges for $\alpha = \theta_c$ and $\omega = 0$. 
Indeed, saying that the integrand stays controllable at $v = \alpha = \theta_c$ 
is equivalent to the same statement for $\frac{1}{\sqrt{f(\theta_c)}}\frac{1}{\sqrt{\theta_c - v}}$. 
But then, at the other side $v = 0$ we get (recall that $F(\theta_c) = 0 = F(0)$):
\[
\frac{1}{\sqrt{- F(v)} } \underset{v\to 0^+}{\sim} \frac{1}{v} \sqrt{- \frac{2}{f'(0)}},
\]
which is not integrable. (Assuming $f'(0) < 0$ for convenience.)
\qed
\end{proof}

\begin{proposition}
 The limit bubble $u_{\theta_c}$ (also known as the ``ground state'') has exponential decay at infinity.
 \label{expodecay}
\end{proposition}
\begin{proof}
 The function $u_{\theta_c}$ satisfies the following equation:
 \[
  \f{\sigma}{2} (u'_{\theta_c})^2 = F(\theta_c) - F(u_{\theta_c}) = - F (u_{\theta_c}).
 \]
  Hence, 
  \[
   \sqrt{\sigma} u'_{\theta_c} = - \sqrt{- 2 F (u_{\theta_c})} \text{ on } \R_+.
  \]
  Moreover, for small $\epsilon$, $\sqrt{-2 F(\epsilon)} = \epsilon \sqrt{-f'(0)} + o (\epsilon)$.
  
  As a consequence, as $u_{\theta_c}$ gets small (at infinity), it is equivalent to the solution of
  \[
   y' = - \sqrt{-f'(0)} y,
  \]
  that is $x \mapsto e^{-\sqrt{-f'(0)} x}$.
  \qed
\end{proof}

{\it Proof of Theorem \ref{thm:invasion} in dimension d=1.}
Let $\alpha\in(\theta_c,1]$, and let us assume that the initial data for system 
\eqref{AllenCahn} satisfies $p(0,\cdot)\geq v_\alpha$ where $v_\alpha$ is the $\alpha$-bubble
defined in Definition \ref{def:valpha}. From Proposition \ref{prop:ualpha},
it suffices to prove that $p(t,\cdot)\to 1$ locally uniformly on $\R$ as $t\to +\infty$.

We first notice that the $\alpha$-bubble $v_\alpha$ is a sub-solution for \eqref{AllenCahn}.
Indeed it is the minimum between the two sub-solutions $0$ and $u_\alpha$. Therefore, by the
comparison principle, if $p(0,\cdot)\geq v_\alpha$, then for all $t>0$, $p(t,\cdot)\geq v_\alpha$.

Then, the proof follows from the ``sharp threshold phenomenon'' for bistable equations, 
as exposed for example in \cite[Theorem 1.3]{DuMat.Convergence}, which we recall below:
\begin{theorem}\cite[Theorem 1.3]{DuMat.Convergence}
Let $\phi_\lambda$, $\lambda>0$ be a family of $L^\infty(\R)$ nonnegative, compactly supported initial data
such that \\
(i) $\lambda\mapsto \phi_\lambda$ is continuous from $\R^+$ to $L^1(\R)$;\\
(ii) if $0<\lambda_1<\lambda_2$ then $\phi_{\lambda_1}\leq \phi_{\lambda_2}$ and 
$\phi_{\lambda_1}\neq \phi_{\lambda_2}$; \\
(iii) $\lim_{\lambda\to 0} \phi_\lambda(x)=0$ a.e. in $\R$.

Let $p_\lambda$ be the solution to \eqref{AllenCahn} with initial data $p_\lambda(0,\cdot)=\phi_\lambda$. 
Then, one of the following alternative holds:\\
(a) $\lim_{t\to \infty} p_\lambda(t,x) = 0$ uniformly in $\R$ for every $\lambda>0$; \\
(b) there exists $\lambda^* \geq 0$ and $x_0 \in \R$ such that 
$$
\lim_{t\to \infty} p_\lambda(t,x) = \left\{ \begin{array}{lll}
0 \qquad &\mbox{ uniformly in }\R  \qquad &(0 \leq \lambda < \lambda^*),  \\
u_{\theta_c}(x-x_0) \qquad &\mbox{ uniformly in }\R  \qquad &(\lambda = \lambda^*),  \\
1 \qquad &\mbox{ locally uniformly in }\R  \qquad &(\lambda > \lambda^*).
\end{array}\right.
$$
\end{theorem}
In our case, we define $\phi_\lambda(x)=v_\alpha(\f{x}{\lambda})$ for $\lambda>0$.
We have $\phi_1=v_\alpha$. Since $v_\alpha$ is a sub-solution to \eqref{AllenCahn},
the solution to this equation with initial data $\phi_1$ stays above $v_\alpha$ for all positive time.
From the alternative in the above Theorem, we deduce that the solution to \eqref{AllenCahn}
with initial data $v_\alpha$ converges to $1$ as time goes to $+\infty$ locally uniformly on $\R$. 
(Indeed, the ground state $u_{\theta_c}$ is bounded from above by $\theta_c < \alpha$.)
By the comparison principle, we conclude that if $p(0,\cdot) \geq v_\alpha$, then
$\lim_{t\to+\infty} p(t,\cdot) = 1$ locally uniformly as $t\to+\infty$.
\qed

\subsection{Comparison of the energy and critical bubble methods}

Our construction of a critical $\alpha$-bubble, inspired by \cite{BarTur.Spatial}, 
holds in dimension $1$. 
In this context we may compare the ``minimal invasion radius'' at level $\alpha$ 
for initial data, given by the two sufficient conditions:
 being above an $\alpha$-bubble (which is the maximum of two stationary solutions), 
 or being above an initial condition with negative energy.

\begin{figure}[h!]
 \includegraphics[width=\textwidth]{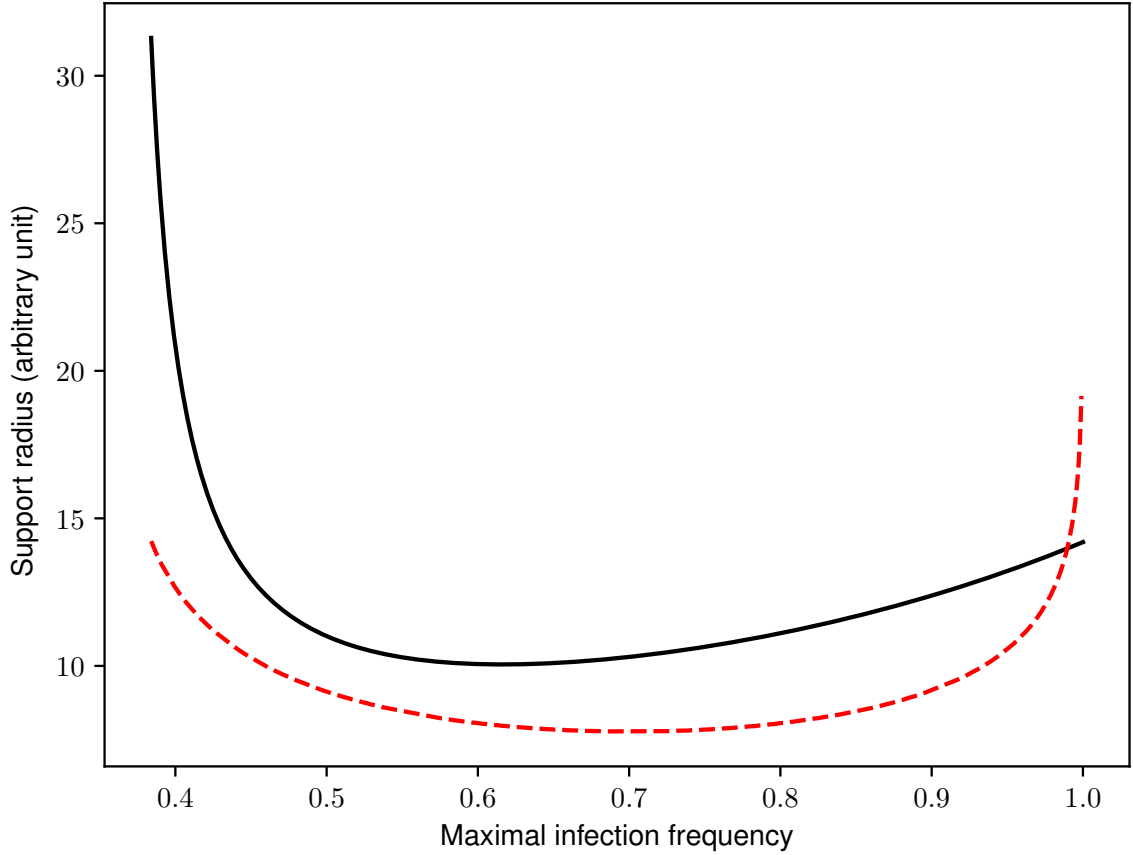}
 \caption{Comparison of minimal invasion radii $R_{\alpha}$ (obtained by energy) in dashed line and $L_{\alpha}$ (obtained by critical bubbles) 
 in solid line, varying with the maximal infection frequency level $\alpha$. The scale is such that $\sigma=1$.}
 \label{fig:comp}
\end{figure}

We first compute the energy of the critical $\alpha$-bubble $v_\alpha$ of  
Definition~\ref{def:valpha},
$$
E[v_\alpha] = \int_{\R} \left(\frac{\sigma}{2} |v_\alpha'|^2 - F(v_\alpha) \right)\,dx.
$$
From Equation~\eqref{Palpha}, we have
$$
E[v_\alpha] = \int_{-L_\alpha}^{L_\alpha} \big( \sigma |v_\alpha'|^2 - F(\alpha)\big) dx
= 2 \int_0^{L_\alpha} \sigma |v_\alpha'|^2 \,dx - 2 L_\alpha F(\alpha).
$$
Performing the change of variable $v=v_\alpha(x)$ we have
$$
\int_0^{L_\alpha} |v_\alpha'|^2 dx = \int^\alpha_0 v_\alpha'(v_\alpha^{-1}(v)) \,dv 
= \frac{1}{\sqrt{\sigma}} \int_0^\alpha \sqrt{2(F(\alpha)-F(v))} dv,
$$
where we use Equation~\eqref{dualpha} for the last equality.
Finally, using the expression of $L_\alpha$ in \eqref{def:lalpha} we arrive at
$$
E[v_\alpha] = 2\sqrt{\sigma} \int_0^\alpha \frac{F(\alpha)-2 F(v)}{\sqrt{2(F(\alpha)-F(v))}}\,dv.
$$

To emphasize the difference between the two sufficient conditions, we observe that
when $\alpha\to \theta_c$, since $F(\theta_c)=0$, we obtain
$$
E[v_{\theta_c}] = 2\sqrt{\sigma} \int_0^{\theta_c} \sqrt{-2F(v)}\,dv.
$$

\begin{lemma}
 The $\alpha$-bubbles $v_{\alpha}$ have positive energy if $\alpha$ is close to $\theta_c$.
\end{lemma}
\begin{proof}
 This follows from continuity of $\alpha \mapsto E[v_{\alpha}]$.
 \qed
\end{proof}

\begin{remark}
In particular, the energy estimate alone does not imply invasiveness of the $\alpha$-bubbles, which justifies the interest of our particular approach in one dimension.
We do not claim that the ``energy'' or the ``bubble'' method is better, 
but we highlight the fact that they do not perfectly overlap.
\end{remark} 
Figure \ref{fig:comp} gives a numerical illustration of the fact that $\alpha$-bubbles give smaller radii at level $\alpha$, except for $\alpha \sim 1$,
and at any rate provide a smaller minimal radius for invasion
when the same parameters as in Figure \ref{fig:reacprofiles} are used.

\section{Specific study of a relevant set of release profiles}
\label{sec:4}

In this section we discuss a specific release protocol, with a total of $N$ mosquitoes divided equally into $k$ locations, in a space of dimension $1$.
It yields a release profile in the set $RP_k^d (N)$ we defined in \eqref{eq:rpset}.

\subsection{Analytical study of the case of a single release}

In the case of a single release ($k=1$), we can easily describe the relationship between 
the mosquito diffusivity $\sigma$ and the total number of mosquitoes to release.
Morally, as long as the mosquitoes diffuse they could theoretically invade (in dimension $1$) by a single release, by introducing a sufficiently large amount of mosquitoes.
This is the object of the next proposition:
\begin{proposition}
Let $p_{\sigma} (\tau) := \f{N G_{\sigma} (\tau)}{N G_{\sigma} (\tau) + N_0}$ be the proportion of released mosquitoes right after introduction (at $t=0^+$), 
where $G_\sigma(\tau):=G_{\sigma,1}(\tau) = \frac{1}{\sqrt{2\pi \sigma}} e^{-\tau^2/2\sigma}$.
 \begin{itemize}
  \item[(i)] If $N \in (0, + \infty)$ is fixed, then there exists a range of values for the diffusivity 
  $S(N) = (0, \sigma_+ (N)]$ such that $\sigma \in S(N)$ if, and only if,
  there exists $\alpha \in (\theta_c, 1]$ such that $p_{\sigma} (\tau) \geq u_{\alpha} (\tau)$ for all $\tau \in (0, \alpha)$. 
  Moreover, $S(N)$ is increasing (with respect to inclusion).
  \item[(ii)] If there exists $\sigma_+$ such that $\sigma \in (0, \sigma_+]$ then there exists $N_m(\sigma_+) \in \R_+$ such that
  if $N \geq N_m (\sigma_+)$ then there exists $\alpha \in (\theta_c, 1]$ such that $p_{\sigma} (\tau) \geq u_{\alpha} (\tau)$ for $\tau \in [0, \alpha]$.
 \end{itemize}
In both cases, evolution in \eqref{AllenCahn} with initial data $p_{\sigma}$ yields invasion by the introduced population.
\label{prop:singleR}
\end{proposition}

Part (i) of Proposition \ref{prop:singleR} asserts that if we fix the total number $N$ of mosquitoes to introduce, single introduction is a failure if diffusivity is too large.
Part (ii) is just the converse viewpoint: if we know estimates on the diffusivity 
(thanks to field experiments like mark-release-recapture for example \cite{Vil.Bayesian}), then we can define a minimal number $N_m$ of mosquitoes to introduce 
at a single location to succeed.

\begin{remark}
If $\alpha \in (\theta_c, 1)$ makes $N G_{\sigma}$ satisfy \eqref{nonext} (``be above the $\alpha$-bubble''), then necessarily (evaluating at $0$ to take the maximum of $G_{\sigma}$), $\alpha \leq \f{N}{N + \sqrt{2 \pi \sigma} N_0}$.
In particular, our under-estimation of the probability is equal to $0$ as soon as 
\[
 N< \sqrt{2 \pi \sigma} N_0 \f{\theta_c}{1 - \theta_c}.
\]
Equivalently, the density of mosquitoes at the center of the single release location $\f{N}{\sqrt{2 \pi \sigma}}$ should exceed $\f{\theta_c}{1 - \theta_c} N_0$ for 
our estimate to prove useful. (If $\theta_c = 0.8$, this is already $4$ times the existing mosquito density. 
If $\theta_c = \f{2}{3}$, then it is only $2$ times; 
in the case of Figure \ref{fig:reacprofiles}, $\theta_c=0.36$ and then the ratio is only $0.56$).
\end{remark}

{\it Proof of Proposition \ref{prop:singleR}.}
Both the introduction profile given by the fraction $\displaystyle\f{N G_{\sigma} (\tau)}{N G_{\sigma} (\tau) + N_0}$ and non-extinction bubbles from 
Theorem \ref{thm:invasion} built by \eqref{Palpha} ($(u_{\alpha} (\tau))$) are symmetric, radial-decreasing functions.
Instead of comparing them, we compare their reciprocals.
We define $T_{\sigma, N}$ such that for all $p \in [0, \alpha]$, 
\[
  \f{N G_{\sigma} \big( T_{\sigma,N} (p)\big)}{N G_{\sigma} \big( T_{\sigma,N} (p)\big) + N_0} = p, 
\]
and $\chi_{\alpha}$ such that $u_{\alpha} (\chi_{\alpha} (p)) = p$.
Respectively, they read
\beq
\bepa
T_{\sigma, N} (p) = \sqrt{2 \sigma} \sqrt{\log \Big( \displaystyle\f{N}{N_0 \sqrt{2 \pi \sigma}} \f{1 - p}{p} \Big)}, 
\\[10pt]
\chi_{\alpha} (p) = \sqrt{\displaystyle\f{\sigma}{2}} \displaystyle\int_ p^{\alpha} \f{dv}{\sqrt{F(\alpha) - F(v)}}.
\label{eq:reciprocals}
\eepa
\eeq

\begin{lemma}
The following equivalence holds
\[
\forall \tau \in \R_+, \,  \f{X_{\tau} (\sigma, N)}{X_{\tau} (\sigma, N) + N_0} \geq  u_{\alpha} (\tau)
\iff
\forall p \text{s. t. } 0 \leq p \leq \alpha, \, \chi_{\alpha} (p) \leq T_{\sigma, N} (p).
\]
This, in turn, rewrites as
\begin{equation}
 \log \big(\f{N}{N_0 \sqrt{2 \pi \sigma}} \big) \geq \big(\int_p^{\alpha} \f{dv}{2 \sqrt{F(\alpha) - F(v)}} \big)^2 - \log \big( \f{1-p}{p} \big), \forall p \in [0, \alpha].
 \label{eq:singleR}
\end{equation}
\end{lemma}
This property follows obviously from \eqref{eq:reciprocals}.

From \eqref{eq:singleR}, we define
\begin{align}
&J_{\alpha} (p) :=  \log (p) - \log(1-p)  + \big( \int_{p}^{\alpha} \f{dv}{2\sqrt{F(\alpha) - F(v)}}\big)^2, \\
& I(\sigma, N) := \log \Big( \f{N}{\sqrt{2 \pi \sigma} N_0} \Big).
\end{align}

For any given $N$, the problem we want to solve amounts at finding couples $(\alpha, \sigma)$ such that 
\beq
\forall p \in [0, \alpha], \, J_{\alpha} (p) \leq I (\sigma, N).
\eeq

\begin{lemma}
  There exists $C > 0$ such that for all $N, \sigma$, 
  there exists $\alpha \in (\theta_c, 1]$ such that $J_{\alpha} \leq I (\sigma, N)$ 
  if, and only if,
  \begin{equation}
   N \geq C N_0 \sqrt{2 \pi \sigma} .
   \label{eq:Nsigma}
  \end{equation}

\label{prop:dim1}
\end{lemma}
\begin{proof}
First, we note that $J_{\alpha} (p) \xrightarrow[p \to 0]{} - \infty$, $J_{\alpha} (\alpha) = \log \big( \f{\alpha}{1 -\alpha} \big)$ and it is continuous.
Moreover,
\[
 J_{\alpha}'(p) = \f{1}{p (1-p)} - \f{1}{\sqrt{F(\alpha) - F(p)}} \int_p^{\alpha} \f{dv}{2\sqrt{F(\alpha) - F(v)}},
\]
and we may compute $\lim_{p \to \alpha} J_{\alpha}' (p) = \f{1}{\alpha (1-\alpha)} - \f{1}{f(\alpha)}$.

Then, we simply introduce 
\[
  j_{\alpha} := \max_{p \in [0, \alpha]} J_{\alpha} (p), \quad j^* := \min_{\alpha \in (\theta_c, 1]} j_{\alpha},
\]
which are well-defined.

Thus, the following is a necessary and sufficient condition for the existence of $\alpha \in (\theta_c, 1]$ such that \eqref{eq:singleR} holds:
\[
  N \geq N_0 \sqrt{2 \pi \sigma} e^{j^*}.
\]
We arrived at \eqref{eq:Nsigma}, upon choosing $C = e^{j^*}$.
\qed
\end{proof}

Finally, Equation~\eqref{eq:Nsigma} gives Proposition \ref{prop:singleR} (i) with $\sigma_+ (N) = \f{e^{-2 j^*}}{2 \pi} \big( \f{N}{N_0} \big)^2$
and Proposition \ref{prop:singleR} (ii) with $N_m = N_0 \sqrt{2 \pi \sigma_+} e^{j^*}$.

\qed

\begin{remark}
 For realistic values of diffusivity and density $N_0$, the expected number of mosquitoes to release is huge, since we may have
 $N_0 \simeq 10^{-2}$, $\sqrt{2 \pi \sigma} \simeq 72$, but $j^* \simeq 38$.
 Here, the model has a clear and crucial conclusion: it is very hard to invade a wide 
 area with a single, localized release.

 Therefore, we must model several releases (whether in time or in space). 
 In the rest of the paper we are going to discuss the case of multiple releases at same time $t=0$.
\end{remark}

\subsection{Equally spaced releases}

Similarly, if we space the $k$ release points regularly in the interval 
$[-L_{\alpha}, L_{\alpha}]$, within a fairly good approximation, 
we obtain the minimal number of mosquitoes to release as
\[
 \widetilde{N} (k, \alpha, \sigma) = \f{N_0 \sqrt{2 \pi \sigma}}{2} \f{\alpha}{1-\alpha} k e^{\f{L_{\alpha}^2}{2 \sigma (k-1)^2}}.
\]
This equation can be used in different ways, just like the above formula \eqref{eq:Nsigma}. If we fix $\sigma$ then we may try to find an optimal $k$ 
(both optimization problems in $\alpha$ and in $k$ must be solved together in this case).
Or fixing $N$, or $N/k$ (number of mosquitoes per release), we can do the same and find the optimal number of releases $k$.

It is straightforward, keeping in mind that $L_{\alpha}$ is proportional to $\sqrt{\sigma}$, 
that the optimal $\alpha$ here merely depends on $k$, not on $\sigma$.
We may introduce
\[
 j^* (k) := \min_{\alpha \in (\theta_c, 1)} \f{\alpha}{1-\alpha} e^{L_{\alpha}^2 / (2 \sigma (k-1)^2)}.
\]
Then, we find the minimal (in view of our sufficient criterion) value $\widetilde{N}^*$ for $\widetilde{N}$:
\begin{lemma}
For $k$ equally spaced releases on the line, there exists an invading release profile with $L^1$ norm:
\begin{equation}
 \widetilde{N}^* (k, \sigma) = N_0 \sqrt{2 \pi \sigma} \f{k}{2} j^* (k).
\end{equation}
\end{lemma}
Then, it becomes an easy numerical task to find the best possible value for $k$.

However, we want to take into account the uncertainties and variability in the release protocol and population fixation.
Namely, the release points might not be exactly equally spaced, so that introducing $\widetilde{N}^*$ mosquitoes would only give some probability of success.
This is what we want to quantify now and shall be addressed in Section \ref{sec:multiplereleases}.

\subsection{Multiple releases: towards a geometric problem}
\label{sec:multiplereleases}

When we sum several Gaussians, the profile is neither symmetric (in general), nor monotone. 
Therefore the previous analytical argument does not apply. 
However, at the cost of fixing $\sigma$ we are left with a simple geometric problem.

\paragraph{First step: fixing $\sigma$ and bounding by level rather than profile.} We assume first that there is no uncertainty on $\sigma$, which is 
taken equal to $\sigma_0$ ($\epsilon = 0$ in \eqref{eq:rpset}).
As a further simplification, we shall not compare the introduction frequency profile to some $\alpha$-bubble (because it is too hard), 
but rather to the very simple upper bound of an $\alpha$-bubble: the characteristic function $\tau \mapsto \alpha \mathbb{1}_{- L_{\alpha} \leq \tau \leq L_{\alpha}}$.

Moreover, we assume that our $k$ release locations $(x_i)_{1 \leq i \leq k}$ are within the compact set $[-L, L]$, for some $L>0$.
As above, we write 
$$G_{\sigma} (y) := \f{1}{\sqrt{2 \pi \sigma}} e^{-y^2 / 2 \sigma} \mbox{, }$$
and $$\mathcal{G} = \f{N}{k} \sum_{i=1}^k G_{\sigma} (\cdot - x_i).$$

We define
\beq
P(\sigma, \f{N}{k}, (x_i)_{1 \leq i \leq k}, L_0, \alpha) := \min_{[-L_{\alpha} +L_0, L_{\alpha} + L_0 ]} \mathcal{G}
\eeq

Then, the probability of success for the release of $N$ mosquitoes in a total of $k$ different sites in $[-L, L]^k$, when they all spread according to $\sigma$ diffusivity, 
and the initial population density was $N_0$, is given by:
\beq
P_k (N, L) = \Pro \Big[ \exists L_0 \in \R, \, \exists \alpha \in (\theta_c, 1), \, P(\sigma, \f{N}{k}, (x_i)_{1 \leq i \leq k}, L_0, \alpha) \geq \f{\alpha}{1 - \alpha} N_0 \Big] \mbox{.}
\eeq
Here, the probability $\Pro$ is taken over all the real $k$-uples $(x_l)_{1 \leq l \leq k}$ such that $-L < x_1 \leq \dots \leq x_k < L$, and $[-L, L]^k$ is equipped with the uniform measure.

\paragraph{Second step: transformation into a geometric problem.}

In order to get a more tractable bound, we make use of the following property:
\begin{proposition}
 Let $(x_i)_{1 \leq i \leq k} \in [-L, L]^k$ with $x_1 \leq \dots \leq x_k$. Let 
$\mathcal{G} = \frac{N}{k}\sum_{i=1}^k G_{\sigma} (\cdot - x_i)$.
 
 If there is $\alpha \in (\theta_c, 1)$ such that 
 $$\f{N}{k} \f{1}{\sqrt{2 \pi \sigma}} \geq \f{\alpha}{1-\alpha} N_0$$ 
 and $1 \leq l < m \leq k$ such that
 \begin{itemize}
  \item[(i)] $\forall l \leq j \leq m-1, \, x_{j+1} - x_j \leq 2 \sqrt{2 \log(2)} \sqrt{\sigma}$,
  \item[(ii)] $ x_m - x_l \geq 2 L_{\alpha}$,
 \end{itemize}
  then 
  $$\f{\mathcal{G}}{\mathcal{G} + N_0} \geq v_{\alpha} \big(\cdot - \f{x_m + x_l}{2} \big) \mbox{ .}$$
  
  \label{prop:geom}
\end{proposition}
We notice that the constant $2 \sqrt{2 \log (2)} \simeq 2.35$ is optimal with this property: if two translated Gaussians centered at $x_0, x_1$ are at a distance $x_1 - x_0 = \lambda \sqrt{\sigma}$, with $\lambda > 2 \sqrt{2 \log (2)}$, then their sum is smaller at $\f{x_0 + x_1}{2}$ than at $x_0$.

\begin{proof}
 This property relies on the simple computation that the sum of two $G_{\sigma}$s, centered at $-h$ and $h$ ($h>0$), is greater than $G_{\sigma} (0)$ on $[-h, h]$ as soon as $h \leq \sqrt{2 \log(2)}\sqrt{\sigma}$.
 Figure \ref{fig:sum} illustrates this property.
   \begin{figure}[h]
   \includegraphics[width=.5\textwidth]{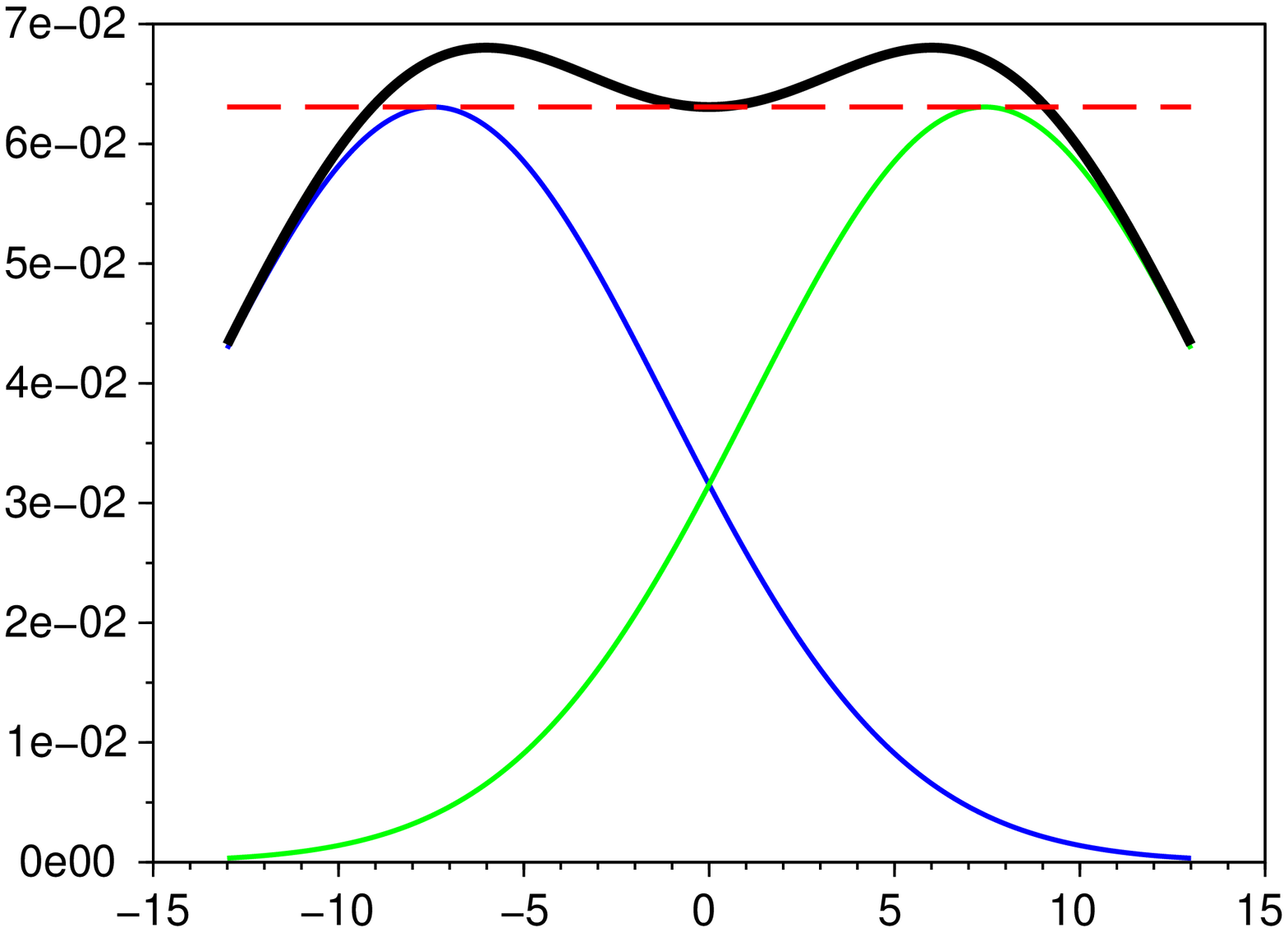}
   \includegraphics[width=.5\textwidth]{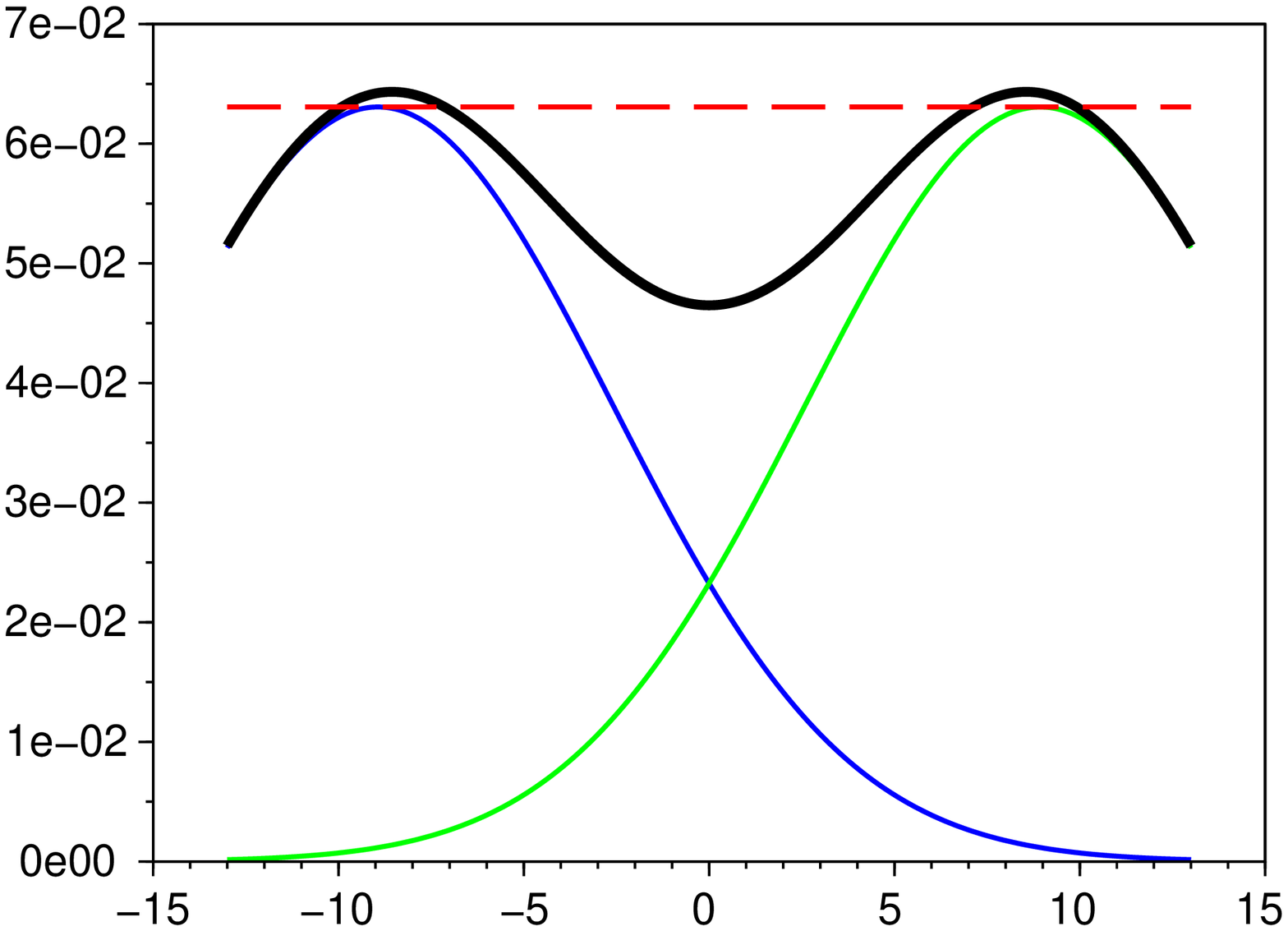}
   \caption{Two $G_{\sigma}$ profiles and their sum (in thick line). The level $G_{\sigma} (0)$ is the dashed line. On the left, $h=\sqrt{2\log(2)\sigma}$. On the right, $h>\sqrt{2\log(2)\sigma}$.}
   \label{fig:sum}
  \end{figure}

Indeed, considering the sum of two Gaussian $G_\sigma$,
 \[
  \xi (x) = \f{1}{\sqrt{2 \pi \sigma}} \Big(e^{-\f{(x + h)^2}{2 \sigma}} +e^{-\f{(x - h)^2}{2 \sigma}} \Big) = 2 e^{-\f{h^2}{2 \sigma}} G_{\sigma} (x)  \cosh ( \f{xh}{\sigma} ).
 \]
  Then, recalling that $\sigma G_{\sigma}'(z) = - z G_{\sigma} (z) $, we compute
  \begin{align*}
   \f{1}{2} e^{\f{h^2}{2 \sigma}} \sigma \xi'(x) &= - x G_{\sigma} (x) \cosh (\f{xh}{\sigma}) + h G_{\sigma} (x) \sinh (\f{xh}{\sigma})\\
   \f{1}{2} e^{\f{h^2}{2 \sigma}} \sigma^2 \xi''(x) &= (h^2 + x^2-\f{1}{\sigma}) G_{\sigma} (x) \cosh (\f{xh}{\sigma}) - 2 h x G_{\sigma} (x) \sinh (\f{xh}{\sigma}).
  \end{align*}
  As a consequence, the sign of $\xi'' (x)$ is that of
  \[
   \gamma(x) := h^2 + x^2 - 2 h x \tanh (\f{xh}{\sigma}) - \f{1}{\sigma}.
  \]
  We notice that $\gamma(0) = h^2 - \f{1}{\sigma}$. Hence, $\xi$ has a local maximum (resp. a local minimum) at $x=0$ if $h < \sqrt{\sigma}$ (resp. $h > \sqrt{\sigma}$).
  Since $\xi(0) = 2 e^{-\f{h^2}{2 \sigma}} G_{\sigma} (0)$, the maximal $h > 0$ that ensures
  $\xi(0) \geq G_{\sigma} (0)$ is $h=h_0 :=\sqrt{2 \log(2) \sigma}$.
  
  Now, we examine the necessary condition $\xi'(x) = 0$ for a local extremum on $(-h, h)$. It implies
  \[
   x = h \tanh (\f{xh}{\sigma}).
  \]
  This is true for $x=0$ (and we have seen the condition on $h - \sqrt{\sigma}$ to have a local extremum indeed).
  Then, there is a solution $x_+ > 0$ if, and only if, $\f{h^2}{\sigma} > 1$, \textit{i.e.} $h > \sqrt{\sigma}$. 
  In this case, $x_+$ is unique (and $x_- := - x_+$ is a solution as well).

  So, for $h=h_0>\sqrt{\sigma}$, we know that $\xi$ has a local minimum at $x=0$, is smooth, has at most one local extremum on $(0, + \infty)$, and goes to $0$ at $+\infty$.
  Hence, this local extremum exists and is a maximum. 
  Therefore (and by symmetry), the minimum of $\xi$ on $(-h, h)$ is attained at $x=0$ or $x=h$.
  Since $h=h_0$, $\xi(h) > \xi(0) = G_{\sigma} (0)$. 
  We deduce that $\xi > G_{\sigma} (0)$ on $(-h, h)$.
 
  We may use this property to prove Proposition \ref{prop:geom}.
  By condition (i) the above lower-bound 
  holds between $x_l$ and $x_m$, and not only between two adjacent locations $x_j, x_{j+1}$.
  Now, the first condition implies that $G_{\sigma} (0) \geq \f{\alpha}{1-\alpha} N_0$.
  Combining these two facts with $x_m - x_l \geq 2 L_{\alpha}$ implies that
  \[
   \f{\mathcal{G}}{\mathcal{G} + N_0} \geq \alpha,
  \]
  on $[x_l,x_m]$ which is an interval of length at least $[-L_{\alpha}, L_{\alpha}]$. 
  Precisely, for all $x \in \R$,
  $$
  \f{\mathcal{G}(x - \f{x_m+x_l}{2})}{\mathcal{G}(x - \f{x_m+x_l}{2})+N_0} \geq \alpha \geq v_{\alpha}(x - \f{x_m+x_l}{2}).
  $$ 
  
 \end{proof}

As a consequence, we may translate the generic inequality \eqref{sucessprobability} into:
 \begin{multline}
  P_k^1(N,(-L,L)) = P_k (N, L) \geq \Pro \Big[ \exists \alpha \in (\theta_c, \f{1}{1+ \f{N_0}{N} k \sqrt{2 \pi \sigma}}), \exists 1 \leq l < m \leq k, \\
  x_m - x_l \geq 2 L_{\alpha} \text{ and } \forall l \leq j \leq m-1, x_{j+1} - x_j \leq 2 \sqrt{2 \log(2)} \sqrt{\sigma} \Big]
  \label{eq:cs1}
 \end{multline}
 
 Then, we define 
 \[
 L^* := \min_{\theta_c < \alpha \leq \f{1}{1+\f{N_0}{N} k \sqrt{2 \pi \sigma}}} L_{\alpha},
  \]
and estimate \eqref{eq:cs1} is equivalent to
\begin{equation}
  P_k (N, L) \geq 
  \Pro \Big[ \exists 1 \leq l < m \leq k, x_m - x_l \geq 2 L^* 
   \text{ and } \max_{l \leq j \leq m-1} (x_{j+1} - x_j) \leq 2 \sqrt{2 \log(2)} \sqrt{\sigma} \Big].
  \label{simpleproba}
\end{equation}
The study of the minimization of $L_\alpha$ with respect to $\alpha$ is discussed further in Appendix.

\begin{remark}
Note that for this estimate, we only consider initial data that are above a characteristic function at level $\alpha$ on an interval of length $2 L_{\alpha}$. This is far from being the optimal way to be above the $\alpha$-bubble $v_{\alpha}$.
\end{remark}

\begin{remark}
It is easy to check that our estimate yields $0$ (no information) as long as $k$ is too small, namely $k \sqrt{2 \log(2)} \sqrt{\sigma} \leq L^*$.
A necessary condition for our estimate not to yield $0$ may read:
\[
 k \geq \f{1}{\sqrt{2 \log(2)}} \min_{\theta_c < \alpha \leq 1} \int_0^{\alpha} \f{dv}{\sqrt{2 \big(F(\alpha) - F(v) \big)}}.
\]
\end{remark}

\paragraph{Specific discussion for $\alpha = \theta_c$.} By Proposition \ref{expodecay}, $u_{\theta_c}$ decays exponentially. As a consequence, no sum of $G_{\sigma}$s may be above it.
This is why this profile cannot be used in our approach (because we consider that introduction profiles should be Gaussian).

\subsection{Analytical computations of the probability of success: recursive formulae}

In order to compute analytically the right-hand-side in \eqref{simpleproba}, we may introduce the following notations:
 \begin{itemize}
  \item $\mathcal{T}_k (u, v)$ is the set of ordered $k$-uples between $u$ and $v$ ($u<v \in \R$), the measure of which is
 \[
  \tau_k (u, v) = \f{(v-u)_+^k}{k!}.
 \]
  \item $\mathcal{C}_k^{\lambda} (u,v) \subseteq \mathcal{T}_k (u,v)$ is the subset of $k$-uples such that $y_1= u$, $y_k = v$ and for all $l \in \llbracket 1, k-1 \rrbracket$, $y_{l+1} - y_l \leq \lambda$. 
    Its measure is denoted $\gamma_k^\lambda(u,v)$.
  \item $\mathcal{B}^{\lambda, R^*}_k (u,v) \subseteq \mathcal{T}_k (u,v)$ is the subset of $k$-uples such that $\exists 1 \leq l < m \leq k, y_m - y_l \geq R^*$ and $\max_{l \leq j \leq m-1} (y_{j+1} - y_j) \leq \lambda$.
    We denote $\beta_k^{L,R^*}(u,v)$ its measure.
 \end{itemize}
 We want to under-estimate the probability of success with $k$ releases in the box $[-L, L]$.
 In view of \eqref{sucessprobability}, it amounts to computing $\f{\beta_k^{\lambda, R^*} (-L, L)}{\tau_k (-L, L)}$. 
 In fact, we get a general recursive formula for $\beta$ in the following proposition.

\begin{proposition}
  Let $k_0 := \llceil \f{R^*}{\lambda} \rrceil + 1$. Then,
  \begin{multline}
   \beta_k^{\lambda, R^*} (-L, \chi) = \sum_{i=k_0}^k \sum_{j=1}^{k-i+1} \int_{-L}^{\chi-R^*} \int_{u+R^*}^{\min (\chi, u + (k-1) \lambda) } \gamma_i^{\lambda} (u, v)
   \\  \Big( \tau_{j-1} \big(-L, u - \lambda \big) - \beta_{j-1}^{\lambda, R^*} \big(-L, u - \lambda \big) \Big) \tau_{k-(i+j-1)} \big( v + \lambda, \chi \big) dv du.
   \label{eq:recur}
  \end{multline}
  \label{recurbeta}
\end{proposition}
\begin{proof}
 The idea is simple: we count each ``positive initial data'', that is an ordered $k$-uple $(y_i)_i$ such that a subfamily satisfies $y_m - y_l \geq R^*$ and $y_{i+1} - y_i \leq \lambda$ in between $l$ and $m$,
 according to its leftmost ``positive sub-family'', which is then taken of maximal length.
 
 We shall use the index $i$ to denote the length of this maximal family 
 (between $k_0$ and $k$), and $j$ its first rank ($1 \leq j \leq k-i+1$).
 Then,
 \beq
  \beta_k^{\lambda, R^*} (-L, \chi) = \int_{[-L, \chi]^k} \mathbb{1}_{\{ y_1 \leq y_2 \leq \dots \leq y_k \}} \mathbb{1}_{\{ (y_1, \dots, y_k) \in \mathcal{B}^{\lambda, R^*}_k (-L, \chi) \}} dy_1 \dots dy_k .
  \label{eq:carac}
 \eeq
 Now, we split:
 \begin{multline}
  \mathbb{1}_{\{ (y_1, \dots, y_k) \in \mathcal{B}^{\lambda, R^*}_k (-L, \chi) \}} = \sum_{i=k_0}^k \sum_{j=1}^{k-i+1} \mathbb{1}_{\{ y_{i+ j - 1} - y_{j} \geq R^* \}} \prod_{l = j}^{j + i - 2} \mathbb{1}_{\{y_{l+1} - y_l \leq \lambda \}}
  \\ \mathbb{1}_{\{ (y_1, \dots, y_{j-1}) \not\in \mathcal{B}^{\lambda, R^*}_{j-1} (-L, \chi) \}} \mathbb{1}_{\{ y_{j} - y_{j-1} > \lambda \}} \mathbb{1}_{\{ y_{i+ j} - y_{i+ j -1} > \lambda \}}.
  \label{eq:decomp}
 \end{multline}
 This identity requires some explanations.
 It comes directly from the partition of $\mathcal{B}$ using maximal leftmost positive sub-family, as described above.
 Then, the term $\mathbb{1}_{\{ (y_1, \dots, y_{j-1}) \not\in \mathcal{B}^{\lambda, R^*}_{j-1} (-L, \chi) \}}$
 simply comes from the definition of $\mathcal{B}$. 
 Since we consider the \textit{leftmost} positive subfamily, no family on its left should be positive.
 Moreover no element on its left can be added, which justifies the $\mathbb{1}_{\{y_j - y_{j-1} > \lambda \}}$.
 Then, we have in addition that for $j > 1$ and $y_j \leq \chi$,
 \[
  \mathbb{1}_{\{ (y_1, \dots, y_{j-1}) \not\in \mathcal{B}^{\lambda, R^*}_{j-1} (-L, \chi) \}} \mathbb{1}_{\{y_{j-1} \leq y_{j} \}} \mathbb{1}_{\{y_j - y_{j-1} > \lambda \}} 
  = \mathbb{1}_{\{ (y_1, \dots, y_{j-1}) \not\in \mathcal{B}^{\lambda, R^*}_{j-1} (-L, y_j - \lambda) \}},
 \]
  with the obvious convention that $\mathcal{B} (u, v) = \emptyset$ if $v < u$.

  In addition, for $i + j - 1 < k$
 \begin{align*}
   \int_{[-L, \chi]^{k-(i+j-1)}} &\mathbb{1}_{\{y_{i+j-1} \leq \dots \leq y_{k} \}} \mathbb{1}_{\{y_{i+j} - y_{i+j-1} > \lambda \}} dy_{i+j} \dots dy_k \\
& = \tau_{k-(i+j-1)}(y_{i+j-1}+\lambda,\chi)   \\
& = \f{ \big( \chi - y_{i+j-1} - \lambda \big)_+^{k-(i+j-1)} }{\big( k-(i+j-1) \big)!}.
  \end{align*}
  
  Combining these results, and using \eqref{eq:decomp} in \eqref{eq:carac} yields
  \begin{multline}
   \beta_k^{\lambda, R^*} (-L, \chi) = \sum_{i=k_0}^k \sum_{j=1}^{k-i+1} \int_{-L}^\chi \dots \int_{x_{i+j-2}}^\chi 
    \mathbb{1}_{\{ y_{j+i-1} - y_j \geq R^* \}} \prod_{l = j}^{j+i-2} \mathbb{1}_{\{ 0 \leq y_{l+1} - y_l \leq \lambda \}}
   \Big( \tau_{j-1} \big(-L, y_j - \lambda \big) -
   \\ \beta_{j-1}^{\lambda, R^*} \big(-L, y_j - \lambda \big) \Big) 
   \tau_{k-(i+j-1)} \big( y_{i+j-1} + \lambda, \chi \big) d y_{j} \dots d y_{i+j-1},
   \label{eq:interm}
   \end{multline}
  with conventions $\tau_0 = 1$ and $\beta_0 = 0$, regardless of their arguments.
  
  We assume $\chi \geq - L + R^*$ (otherwise $\beta_k^{\lambda,R^*} (-L, \chi) = 0$).
  Using the notation $\gamma$ we introduced, Equation~\eqref{eq:interm} simplifies again into:
  \begin{multline*}
   \beta_k^{\lambda, R^*} (-L, \chi) = \sum_{i=k_0}^k \sum_{j=1}^{k-i+1} \int_{-L}^{\chi-R^*} \int_{u+R^*}^{\min (\chi, u + (k-1) \lambda) } \gamma_i^{\lambda} (u, v)
   \\  \Big( \tau_{j-1} \big(-L, u - \lambda \big) - \beta_{j-1}^{\lambda, R^*} \big(-L, u - \lambda \big) \Big) \tau_{k-(i+j-1)} \big( v + \lambda, \chi \big) dv du,
  \end{multline*}
  where $u$ stands for $y_j$ and $v$ for $y_{i+j-1}$.
  This is our recursive formula \eqref{eq:recur}.
\end{proof}

Now, we may give an explicit formula $\gamma^{\lambda}_i (u, v)$.
We should notice that by definition,
\[
 \gamma^{\lambda}_{i+2} (u, v) = \int_u^{u + \lambda} \int_{u_1}^{u_1 + \lambda} \dots \int_{u_{i-1}}^{u_{i-1} + \lambda} \mathbb{1}_{v \geq u_i \geq v-\lambda} du_i \dots du_1,
\]
that is
\begin{equation}
 \gamma^{\lambda}_{i+2} (u, v) = \int_u^{u + \lambda} \gamma_{i+1}^{\lambda} (u_1, v) du_1.
\label{recurgamma}
\end{equation}

Hence, we deduce the recursive formula,
\begin{lemma}
For all $i, \lambda, u, v$ as above,
\begin{equation}
\gamma_{i+2}^{\lambda} (u, v) = \lambda^i + \sum_{k=1}^{i+1} \f{(-1)^k}{i!} \Big( \binom{i}{k-1} \big( v - u - k \lambda)_+^i + 
(-1)^{i+1} \binom{i-1}{k-1} \big( k \lambda - (v-u) \big)_+^i \Big).
\label{recurgamma2}
\end{equation}
\end{lemma}
\begin{proof}
 Obviously, $\gamma_2^\lambda (u, v)= \mathbb{1}_{v \geq u \geq v-\lambda}$ and we deduce from \eqref{recurgamma}
 \[
  \gamma_3^{\lambda} (u, v) = \lambda + \big ( v - u - 2 \lambda \big)_+ - \big(\lambda - (v-u) \big)_+ - \big( v-u -\lambda \big)_+
 \]
  Then, using \eqref{recurgamma} again proves \eqref{recurgamma2} by induction.
\end{proof}

\begin{remark}
\label{hatL}
For $k < 2 k_0$, formula \eqref{eq:recur} simplifies a lot for it is no longer recursive.
It enables us to compute $\beta_{k_0}^{\lambda, R^*} (-L, L)$.

\beq
 \beta_{k_0}^{\lambda, R^*} (-L, L) = \int_{-L}^{L - R^*} \int_{u+R^*}^{\min(L, u + (k_0-1) \lambda)} \gamma_{k_0}^{\lambda} (u, v) dv du.
\label{eq:k0proba}
\eeq
Then by \eqref{recurgamma2} we know $\gamma_{k_0}^{\lambda} (u, v)$. 
With the change of variables $w = v + u$, when $L > - L + (k_0 - 1) \lambda$, equation \eqref{eq:k0proba} becomes
\begin{multline}
 \beta_{k_0}^{\lambda, R^*} (-L, L) =
 \int_{-L}^{L - (k_0 - 1) \lambda} \int_{R^*}^{(k_0 - 1) \lambda} \Big( \lambda^{k_0-2} +
 \\ \sum_{k=1}^{k_0 - 1} \f{(-1)^k}{(k_0 - 2) !} \big( \binom{k_0 - 2}{k - 1} (w- k\lambda)_+^{k_0-2} + (-1)^{k_0 - 1} \binom{k_0 - 3}{k-1} (k \lambda - w)_+^{k_0-2} \big) \Big) dw du
 \\ + \int_{L - (k_0 - 1) \lambda}^{L - R^*} \int_{R^*}^{L - u} \gamma_{k_0}^{\lambda} (u, u+w) dw du.
 \label{bigRHS}
\end{multline}

Clearly, the first integral in the right-hand side of \eqref{bigRHS} may be written as
\[
 \big( 2 L - (k_0 - 1) \lambda \big)f_1 (\lambda, R^*),
\]
where $f_1$ does not depend on $L$.
With the change of variables $z = L - u$, the second term in the right-hand side  of \eqref{bigRHS} becomes
\begin{multline*}
f_2 (\lambda, R^*) := \int^{(k_0 - 1) \lambda}_{R^*} \int_{R^*}^{z} \Big( \lambda^{k_0-2} + \sum_{k=1}^{k_0-1} \f{(-1)^k}{(k_0 - 2)!}  \big( \binom{k_0-1}{k-1} (w- k\lambda)_+^{k_0-2} \\+ 
 (-1)^{k_0-1}\binom{k_0 - 3}{k-1} (k \lambda - w)_+^{k_0-2}\big) \Big) dw dz.
\end{multline*}
In particular, it appears that it does not depend on $L$. (Recall that by definition, $k_0~=~\llceil \f{R^*}{\lambda} \rrceil~+~1$). 

For $\chi \in (-L + R^*, -L + (k_0 - 1) \lambda)$, we can compute similarly
\[
	\beta_{k_0}^{\lambda, R^*} (-L, \chi) = \int_{R^*}^{\chi - (-L)} \int_{R^*}^z \gamma_{k_0}^{\lambda} (0, w) dw dz,
\]
and notice that our expressions are consistent since
\[
	\beta_{k_0}^{\lambda, R^*} (-L, -L + (k_0 - 1 )\lambda) = \int_{R^*}^{-L + (k_0 - 1) \lambda - (-L)} \int_{R^*}^z \gamma_{k_0}^{\lambda} (0, w) dw dz = f_2 (\lambda, R^*).
\]
All in all, $\beta_{k_0}$ is expressed as follows:
\begin{equation}
\beta_{k_0}^{\lambda, R^*} (- L, \chi) =  \left\{ \begin{array}{l} 0 \text{ if } \chi + L \leq R^*
\\[10pt]
\displaystyle\int_{R^*}^{\chi - (-L)} \int_{R^*}^z \gamma_{k_0}^{\lambda} (0, w) dw dz \text{ if } \chi  + L \in (R^*, (k_0 - 1 ) \lambda),
\\[10pt]
\big(\chi + L - (k_0 - 1 )\lambda ) \big) f_1  + f_2  \text{ if } \chi +L > (k_0 - 1 ) \lambda
\end{array} \right.
\end{equation}
(This is an affine function for $\chi + L > (k_0 - 1 ) \lambda$, with pent $f_1 (\lambda, R^*)$).

Then, we obtain a bound on the probability of success with $k_0$ (the minimal number of) releases after dividing by $\tau_{k_0} (-L, L)$ :
\[
 P_{k_0} (L) \geq \f{\beta_{k_0}^{\lambda, R^*}}{\tau_{k_0}} (-L, L) = \f{k_0 !}{(2L)^{k_0}} \big((2L - (k_0 - 1) \lambda)f_1 (\lambda, R^*) + f_2 (\lambda, R^*) \big).
\]

In particular, we see that this underestimation of the success probability is increasing and then decreasing, and thus reaches a unique maximum at $L = \widehat{L}$.

We find 
\[
  2 \widehat{L} = \lambda \big( \llceil \f{R^*}{\lambda} \rrceil + 1 \big) - \f{k_0}{k_0-1} \f{f_2 (\lambda, R^*)}{f_1 (\lambda, R^*)}.
\]
We may note that introducing the non-negative and non-decreasing function 
\[
\Gamma^{\lambda, R^*}_k (z) := \int_{R^*}^z \gamma^{\lambda}_k (0, w) dw 
\]
we get
\begin{align*}
 f_1 (\lambda, R^*) &= \Gamma^{\lambda, R^*}_{k_0} \big( (k_0 - 1) \lambda \big), \\
 f_2 (\lambda, R^*) &= \int_{R^*}^{(k_0-1) \lambda} \Gamma^{\lambda, R^*}_{k_0} (z) dz.
\end{align*}
As a consequence, $f_2 \leq \big( (k_0-1)\lambda - R^* \big) f_1$ and thus
\[
 2 \widehat{L} \geq \f{k_0}{k_0 - 1} R^*.
\]

\end{remark}

\begin{remark}
 Back to problem \eqref{simpleproba}, we recover the problem of estimating $\beta$ with the notations of Proposition \ref{recurbeta} through a simple change of variables.
 We divide all positions ($x_1, \dots, x_k$) by $\sqrt{2 \sigma}$.
 Then in the right-hand side of \eqref{simpleproba} we replace  $2 L^*$ by 
  $$R^* := \min_{\alpha} \int_0^{\alpha} \f{dv}{\sqrt{F(\alpha) - F(v)}}\mbox{ ,}$$ and $2 \sqrt{2 \log(2) \sigma}$ by $\lambda := 2 \sqrt{\log(2)}$.
  This was done in order to simplify computations.
  Moreover, it shows that the success probabilities do not depend on diffusivity.
  In fact, scaling in $\sigma$ as we did merely amounts at choosing a space scale such that $\sigma = 1$.
  Even though probabilities themselves do not make $\sigma$ appear, one must keep in mind
  that the corresponding release protocols (including the space between release points or the size of the release box) are proportional to $\sqrt{\sigma}$.
\end{remark}

\section{Numerical results}
\label{sec:num}

Now, we present some numerical results we obtained on this set of release profiles.
Numerical simulations confirm the intuition of Proposition \ref{prop:coupon}.
Our under-estimation is not very bad. Indeed, as one increases the number of release points ($k$) in a fixed perimeter, with a fixed number of mosquitoes per release, 
then our under-estimation of the probability of success converges to $1$.

\begin{figure}[h]
\centering
 \includegraphics[width=\textwidth]{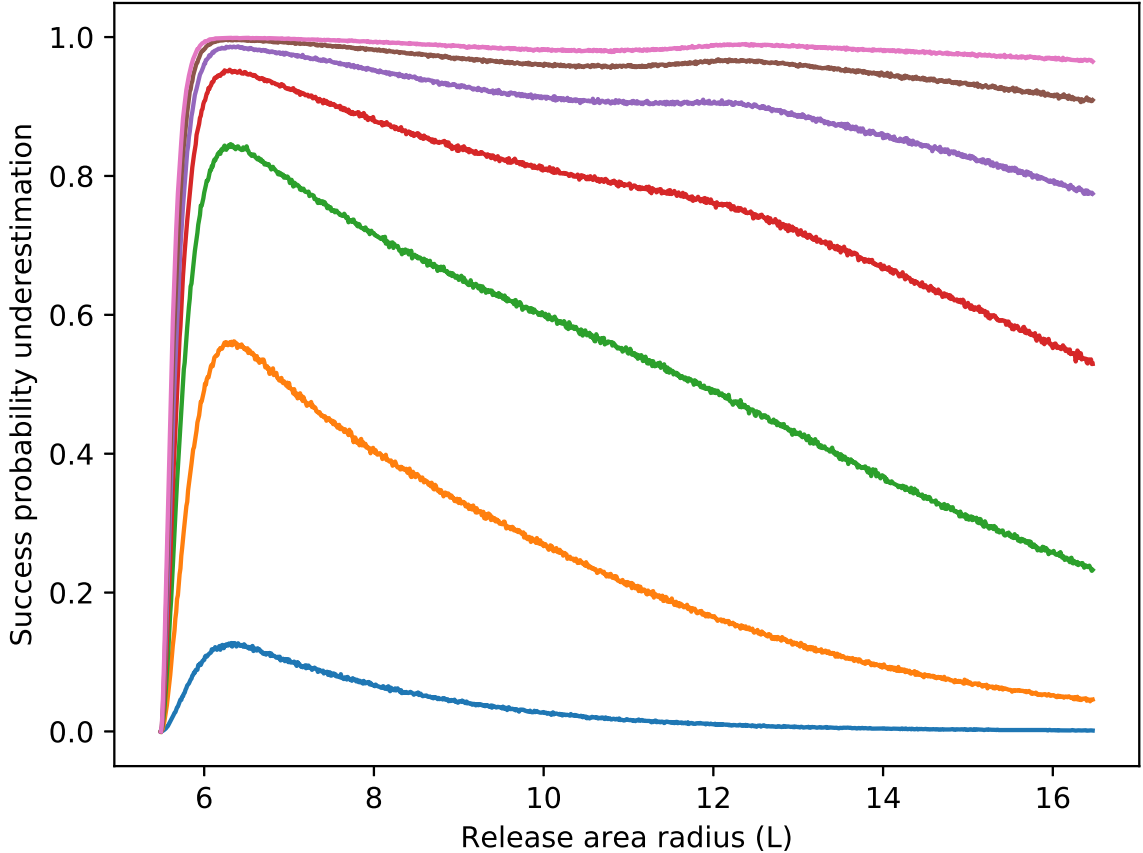}
 \caption{Under-estimation $\beta^{\lambda, R^*} (-L, L)$ of introduction success probability for $L$ ranging from $R^*/2 = 5.49$ to $3 R^*/2 = 16.47$. 
 The seven curves correspond to increasing number of release points. (From bottom to top: $20$ to $80$ release points).}
 \label{fig:80-40-20}
\end{figure}

\begin{figure}[h]
\centering
 \includegraphics[width=\textwidth]{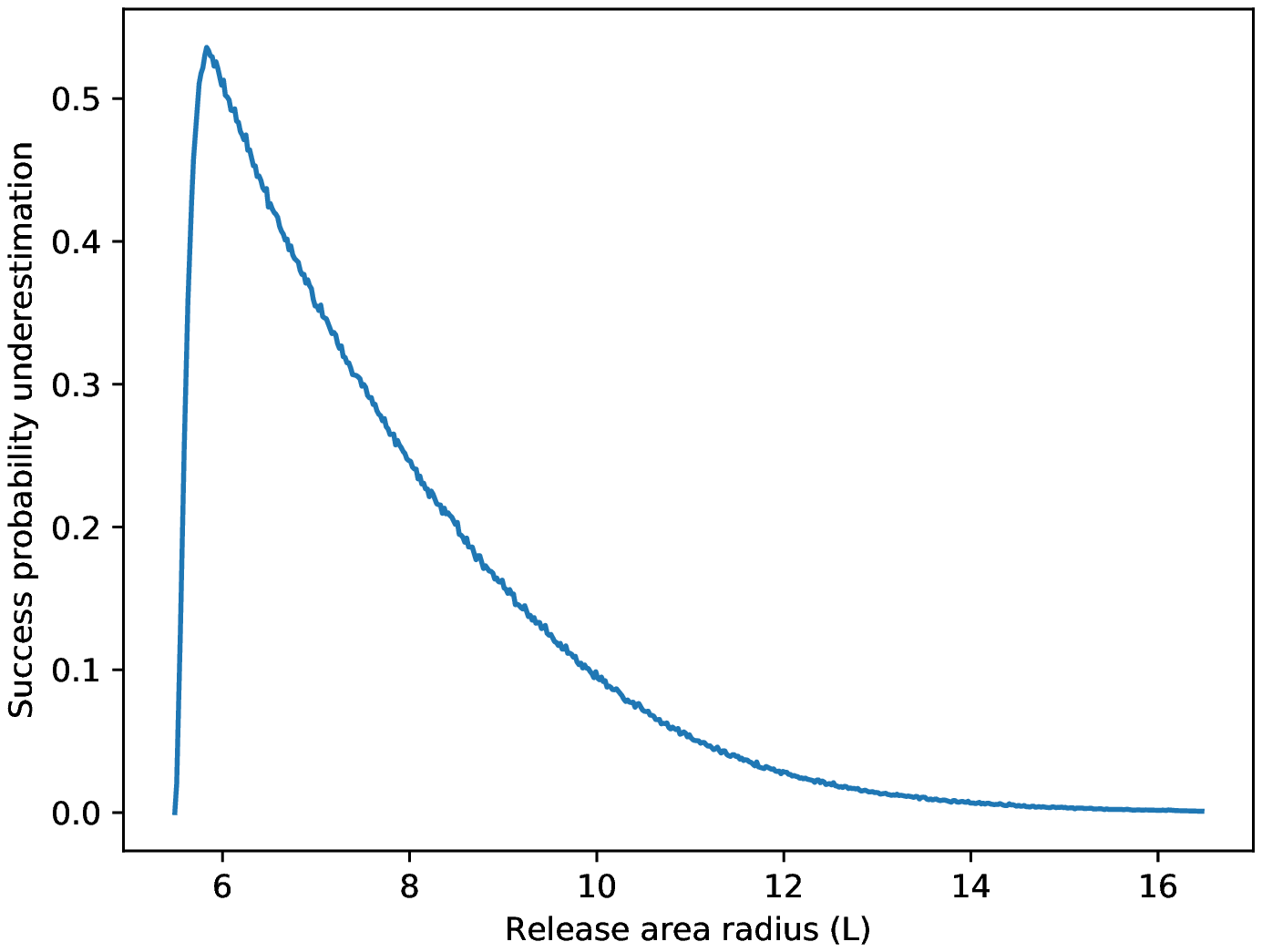}
 \caption{Effect of losing the constant $2 \sqrt{2 \log(2)}$ in Proposition \ref{prop:geom}: under-estimation $\beta^{\lambda, R^*} (-L, L)$ of introduction success probability for $L$ ranging from $R^*/2 = 5.49$ to $3 R^*/2 = 16.47$, with $80$ release points.}
 \label{fig:without}
\end{figure}

Figure \ref{fig:80-40-20} shows the probability profile as a function of the size $L$ of the release box, for $20$, $40$ and $80$ release points. (Here, $R^* = 10.981$, $\lambda = 1.665$ and thus $k_0 = 8$.)
The curves are obtained by a simple Monte-Carlo method.
They lead to the appearance an optimal size for the release box ($6.3$ in this example), that does not seem to depend on the number of release points between $20$ and $80$.

However, for small (relatively to $k_0$) numbers of releases, the probabilities are very small.
In the case of $10$ release points, the maximal probability we find is about $1. 10^{-5}$.

Our numerical values are somehow consistent with field experiments 
(typically, the space between release points is less than $\lambda \sqrt{2 \sigma}$, 
which is about $68 \mathrm{m}$, and the optimal box size is approximately equal to 
$6.3 \times \sqrt{2 \sigma} \simeq 257 \mathrm{m}$).

The factor $2 \sqrt{2 \log(2)}$ is crucial with this respect. 
Losing it changes $\lambda$ from $2\sqrt{ \log(2)} \simeq 1.665$ to $1/\sqrt{2} \simeq 0.707$ and makes $k_0$ (``the minimal theoretical number of releases to make our 
under-estimation of the probability of success positive'') 
increase from $8$ to $17$.
We show in Figure \ref{fig:without} the probability profile for $80$ releases in this case, to illustrate the loss with this ``worse'' geometric estimation.
It culminates at around $50 \%$ only and is comparable with the green curve (for $40$ release points) of Figure \ref{fig:80-40-20}.

\section{Conclusion and Perspectives}
\label{sec:conclusion}

We considered spatial aspects of a biological invasion mechanism associated to release programs and their uncertainty.
We validated the framework in the one-dimensional case, and the two-dimensional case is the natural extension. 

Two difficulties must be tackled in higher dimensions.
First, the radially-symmetric ``$\alpha$-bubbles'' may still exist, but we no longer have an exact formula like \eqref{def:lalpha} for their support.
Second, the geometric problem underlying our estimation gets harder, but not impossible to manage.
To deal with it, we need an analogue of Proposition \ref{prop:geom} in order to get a lower bound for a sum of Gaussians in two dimensions.

An interesting feature of the approach we introduced is that it can be extended to cases when neither sub-solutions nor geometric properties are available.
Heuristically, we need first a criterion to tell us if a given initial data belongs to a ``set of interest''.
Second, we need to put a probability measure on the set of ``feasible initial data''.
Combining these, we compute the probability that the criterion is satisfied.
This probability gives an insight into the role any given aspect of the release protocol plays.

We used a sufficient condition for invasion, the criterion from Theorem~\ref{thm:invasion}.  
However, we proved that our under-estimation of probability is rather good: in particular, it converges to $1$ when the number $k$ of releases goes to $\infty$.
This fact is the object of Proposition~\ref{prop:coupon},  holds true in any dimension, and is supported by numerical simulations in dimension $1$.

We have always considered a homogeneous ``context of introduction'', so that the stochasticity would only affect the release process itself.
Another natural continuation of this work, trying to go further into spatial stochasticity for release protocols, is the use of other stochastic parameters, such as the diffusion process (here it is given by a deterministic diffusivity $\sigma$), or the local carrying capacity. We let this open for further research.

Some other questions remain open. For instance: 
in one dimension, we considered releases in $[-L, L]$. We know that if $2 L < L^*$ then our condition in the right-hand side of \eqref{simpleproba} is zero.
On the other hand, this right-hand side goes to $0$ as $L \to + \infty$.
This suggests that there exists a (non-necessarily unique) size $\widehat{L}$ that maximizes this right-hand side.
Back to \eqref{eq:k0proba}, we obtained in Remark \ref{hatL} a lower bound for $\widehat{L}$ in this case:
\beq
\widehat{L} \geq R^* \f{ 1 + \llceil \f{R^*}{\lambda} \rrceil}{\llceil \f{R^*}{\lambda} \rrceil} .
\eeq
It is a numerical conjecture that the optimal value of $L$ is close to $\f{1}{2} ( \lambda + R^*)$ for any $k$.
For this particular protocol feature (the optimal size of the release area), 
our approach already provides an interesting indication which 
- to the best of our knowledge - has not been used in previous release experiments.

As a possible follow-up to this work, one can set up several optimization problems.
 First, on a purely theoretical side, how to optimize the threshold functions in Theorem \ref{thm:invasion} with respect to a cost functional such as the $L^1$ norm (for the total number of released mosquitoes)?
 Then, if we fix a cost, how to maximize the under-estimated probability of success with respect to the size of the release area?
 Ultimately, how to optimize a release protocol (playing on the probability law of the release profiles space)?

\appendix

\section*{Appendix: Uniqueness of the minimal radius}
\label{uniqueness}

In this appendix we investigate sufficient conditions for the uniqueness of a minimal radius among the $\alpha$ bubbles we constructed in Section \ref{sec:1D}.
More precisely, we establish the number of bubbles of a given radius (which is typically $2$). General results in any dimension on the exact multiplicity of solutions for such problems (semilinear elliptic Dirichlet problems) have been obtained in \cite{OuyShi.ExactI} and \cite{OuyShi.ExactII}, so in essence the results below are not new and are even contained in the cited articles. 

However we emphasize that our proof, limited to dimension $1$, uses very simple arguments and even provides an equivalent formulation of the problem in terms of a single real function $h$ built from $f$ and $F$, see formula \eqref{eq:hExact} below.

Let $f \in \mathcal{C}^2 ([0, 1], \R)$ be a bistable function in the sense of \eqref{hyp:bistab} 
and $F (x) = \int_0^x f(y) dy$ its antiderivative as introduced in \eqref{def:F}.

We make the following assumptions:
\begin{align}
 & f'(0) < 0, \quad f'(\theta) > 0, \quad f'(1) < 0, \tag{B0} \label{bist:0} \\
 &F(1) > 0, \tag{B1} \label{bist:1} \\
 & \forall x \in [0, 1], \quad \big( f'(x) + x f''(x) \big) f(x) \leq x \big( f'(x) \big)^2. \tag{B2} \label{bist:2}
\end{align}
Under assumption \eqref{bist:1}, there exists a unique $\theta_c \in (\theta, 1)$ such that $F(\theta_c) = 0$.
We introduce 
\beq
g(x) := x f'(x) / f(x).
\eeq

\begin{lemma}
Under assumption \eqref{bist:0}, \eqref{bist:2},
$g$ is decreasing on $[0, \theta)$ and on $(\theta, 1]$.
  In addition, $g(0) = 1$, $g(\theta_-) = - \infty$, $g(\theta_+) = + \infty$ and $g(1) = - \infty$.
As a consequence, there exists a unique $\alpha_1 \in (\theta, 1)$ such that
\[
 g(\alpha_1) = 1.
\]
\label{lem:varg}
\end{lemma}
\begin{proof}
 straightforward computation.
\end{proof}

We add the following assumption:
\beq
\forall \alpha > \max(\theta_c, \alpha_1), \quad F(\alpha) \big( f(\alpha) + \alpha f'(\alpha) \big) \leq \alpha \big( f(\alpha) \big)^2.
\tag{B3}
\label{bist:3}
\eeq

Now, we recall the $\alpha$-bubble radius, as introduced before, for $\alpha \in (\theta_c, 1]$:
 \[
  L_{\alpha} = \sqrt{\sigma} \int_0^\alpha \f{dv}{\sqrt{2 \big( F(\alpha) - F(v) \big)}}.
 \]

\begin{proposition}
\label{prop:uniqueness}
Under conditions \eqref{bist:0}, \eqref{bist:1}, the bistable (in the sense of \eqref{hyp:bistab}) function $f$ is such that
$L_{\alpha}$ reaches its minimum on $(\theta_c, 1]$ (which is well-defined) at points in $(\theta_c, 1)$.

If in addition \eqref{bist:2}, \eqref{bist:3} hold,
 then there exists a unique $\alpha_0 \in (\theta_c, 1)$ such that $$L_{\alpha_0} = \min_{\alpha} L_{\alpha}.$$

\end{proposition}
\begin{remark}
 Although assumptions \eqref{bist:0} and \eqref{bist:1} are very general, \eqref{bist:2} and \eqref{bist:3} are debatable.
 They yield a simple sufficient condition for uniqueness of minimum (which is the object of Proposition \ref{prop:uniqueness}),
 but are by no means necessary to get it.
 We expect that they can be refined and improved in order to get uniqueness for a wider class of bistable functions.
 
 However, typical reaction terms in the setting of \textit{Wolbachia} easily satisfy these assumptions,
 and since they are easy to check on any given reaction term, we are happy with them.
\end{remark}

\begin{proof}
 Without loss of generality we assume $\sqrt{\sigma} = \sqrt{2}$ to get rid of the constant.
 From \eqref{def:lalpha}, we deduce the equivalent expression:
 \begin{align*}
  L_{\alpha} &= \int_0^{\alpha} \big( \f{1}{\sqrt{F(\alpha) - F(v)}} - \f{1}{\sqrt{f(\alpha)(\alpha - v)}}\big) dv + \int_0^{\alpha} \f{dv}{\sqrt{f(\alpha)(\alpha - v)}}\\
  &= \f{1}{\sqrt{f(\alpha)}} \Big( \int_0^{\alpha} \big( \f{\sqrt{f(\alpha)}}{\sqrt{F(\alpha) - F(v)}} - \f{1}{\sqrt{\alpha - v}} \big) dv + 2 \sqrt{\alpha} \Big)
 \end{align*}
  Hence
 \[
  \f{d}{d \alpha} L_{\alpha} = \f{1}{\sqrt{\alpha f(\alpha)}} + \f{1}{2 \sqrt{f(\alpha)}} \int_0^{\alpha} \Big(\f{1}{(\alpha - v)^{3/2}} - \big( \f{f(\alpha)}{F(\alpha) - F(v)} \big)^{3/2} \Big) dv,
 \]
 which is a continuous function from $(\theta_c, 1)$ to $\R$. It is easily seen that $\f{d}{d\alpha} L_{\alpha}$ goes to $- \infty$ as $\alpha \to \theta_c^+$, and to $+ \infty$ as $\alpha \to 1^-$ (recalling $f(1) = 0$).
 Therefore, we know that $L_{\alpha}$ reaches its minimum (which is well-defined) at points strictly in the interior of $(\theta_c, 1)$.
 This is the first point of Proposition \ref{prop:uniqueness}.
  
 Then, $\f{d}{d \alpha} L_{\alpha} = 0$ if and only if
 \[
  \f{1}{\sqrt{\alpha}} + \f{1}{2} \int_0^{\alpha} \Big( \f{1}{(\alpha - v)^{3/2}} - \big( \f{f(\alpha)}{F(\alpha) - F(v)} \big)^{3/2} \Big)dv = 0.
 \]
  
  For $\alpha \in (\theta_c, 1)$, we introduce
  \begin{equation}
   h(\alpha) := \int_0^1 \Big( \f{1}{(1 - w)^{3/2}} - \big( \f{\alpha f(\alpha)}{F(\alpha) - F(\alpha w)} \big)^{3/2} \Big) dw.
   \label{eq:hExact}
  \end{equation}

  Then $\f{d}{d \alpha} L_{\alpha} = 0$ if and only if $h(\alpha) = -2$.
  In addition, $h(\theta_c) = - \infty$ and $h(1)= +\infty$ are well-defined by continuity.
  
  We compute 
  \begin{multline}
   h'(\alpha) = - \f{3}{2} \int_0^1 \f{ \big( \alpha f(\alpha) \big)^{1/2}}{ \big(F(\alpha) - F(\alpha w) \big)^{5/2}} \Big( \big( f(\alpha) + \alpha f'(\alpha) \big) \big(F(\alpha) - F(\alpha w) \big) 
   \\ - \alpha f(\alpha) \big( f(\alpha) - w f(\alpha w) \big) \Big) dw,
  \end{multline}
  and introduce
  \[
   z(\alpha, w) := \big( f(\alpha) + \alpha f'(\alpha) \big) \big(F(\alpha) - F(\alpha w) \big) - \alpha f(\alpha) \big( f(\alpha) - w f(\alpha w) \big).
  \]
  Now, we are going to prove that under conditions \eqref{bist:2}, \eqref{bist:3}, for all $\alpha \in (\theta_c, 1]$, $w \in [0, 1]$, 
  \[
   z(\alpha, w) \leq 0,
  \]
  with strict inequality almost everywhere. 
  First, we notice that $z(\alpha, 1) = 0$ and 
  \[
  z(\alpha, 0) = F(\alpha) \big( f(\alpha) + \alpha f'(\alpha) \big) - \alpha f(\alpha)^2.   
  \]

  Then we compute
  \begin{align*}
   \p_w z &= - \alpha f(\alpha w) \big( f(\alpha) + \alpha f'(\alpha) \big) + \alpha f(\alpha) f(\alpha w) + \alpha^2 w f(\alpha) f'(\alpha w) \\
   &= \alpha^2 w f(\alpha) f'(\alpha w) - \alpha^2 f(\alpha w) f'(\alpha).
  \end{align*}
  Now, denoting $g(x) = x f'(x) / f(x)$, we get
  \beq
   \p_w z = \alpha f(\alpha w) f(\alpha) \big( g(\alpha w) - g(\alpha) \big).
   \label{derivz}
  \eeq
  We are going to make use of the assumptions on $f$ and of equation \eqref{derivz} to prove that $z \leq 0$.

  Recall that there exists a unique $\alpha_1 \in (\theta, 1)$ such that $g(\alpha_1) = 1$.
  If $\alpha \leq \alpha_1$, then for all $w \in [0, \alpha/\theta)$, $g(\alpha w) \leq g(\alpha)$
  while for all $w \in (\alpha / \theta, 1]$, $g(\alpha w) \geq g(\alpha)$ (these facts are stated in Lemma \ref{lem:varg}).
  
  Hence $w \mapsto z (\alpha, w)$ is increasing on $[0, \alpha / \theta]$ and on $[\alpha / \theta, 1]$.
  Since $z(\alpha, 1) = 0$, it implies that $z \leq 0$.
  
  Now, if $\alpha > \alpha_1$, there exists a unique $\beta(\alpha) \in (0, \theta)$ such that
  $g(\beta(\alpha)) = g(\alpha)$.
  In this case, if $w \in [0, \alpha / \beta(\alpha)] \cup (\theta, 1]$, $g(\alpha w) \geq g(\alpha)$.
  If $w \in (\alpha / \beta(\alpha), \theta)$, then $g(\alpha w) < g(\alpha)$.
  Hence, $\p_w z \leq 0$ on $[0, \beta(\alpha)/\alpha]$ and $\p_w z \geq 0$ on 
  $[\beta(\alpha) / \alpha, 1]$.
  It implies that $h \leq 0$ if, and only if, $z(\alpha, 0) \leq 0$ for all $\alpha > \alpha_0$.
  This is assumption \eqref{bist:3}.
  
  All in all, we proved that $z \leq 0$ for all $\alpha, w$.  
  Hence $h'(\alpha) > 0$, and there exists a unique $\alpha_0 \in (\theta_c, 1)$ such that $h(\alpha_0) = -2$.
  
  We conclude that $L_{\alpha}$ is decreasing on $(\theta_c, \alpha_0)$ and increasing on $(\alpha_0, 1]$.
  Hence $\alpha_0$ is the unique minimum point of $L_{\alpha}$.  
  
  \end{proof}

\section*{Acknowledgements}
The authors acknowledge partial support from Capes/Cofecub project Ma-833 15 “Modeling  innovative  control  method  for  Dengue  fever” and from the Programme Convergence Sorbonne Universit\'{e}s / FAPERJ “Control  and  identification  for  mathematical  models  of  Dengue epidemics”.
MS and NV acknowledge partial funding from the ANR blanche project Kibord: ANR-13-BS01-0004 funded by the French Ministry of Research, from the Emergence project from Mairie de Paris, Analysis and simulation of optimal shapes - application to lifesciences and from Inria, France and CAPES, Brazil (processo 99999.007551/2015-00), in the framework of the STIC AmSud project MOSTICAW.
JPZ was supported by CNPq grants 302161/2003-1 and 474085/2003-1, by
FAPERJ through the programs {\em Cientistas do Nosso Estado}, and by
the Brazil-France cooperation agreement.

\bibliographystyle{alpha}
\bibliography{biblio}   

%
%

\end{document}